\documentclass[12pt]{article}
\usepackage{amsmath}
\usepackage{amsfonts}
\usepackage{amssymb}
\usepackage{xcolor}
\usepackage{mathrsfs}
\usepackage{enumerate}
\usepackage{cite}
\usepackage{amsthm}

\usepackage{braket}
\theoremstyle{plain}
\newtheorem{Thm}{Theorem}[section]
\newtheorem{Cor}{Corollary}[section]
\newtheorem{Lem}{Lemma}[section]
\newtheorem{Pro}{Proposition}[section]

\newtheorem{Rek}{Remark}[section]
\theoremstyle{definition}

\usepackage{times}

 \usepackage{hyperref}
\hypersetup{
    colorlinks=true,
    linkcolor=blue,
    filecolor=blue,      
    urlcolor=blue,
    citecolor=red,
}
\def\R{\mathbb{R}}

\def\e{\varepsilon}
\def\d{\mathrm{d}}
\def\abs#1{|{#1}|}
\def\norm#1{\|{#1}\|}

\def\supp{\mbox{supp}}
\def\half{{\frac{1} {2}}}

\usepackage{authblk}
\title{\bf Qualitative analysis on logarithmic Schr\"odinger equation
with general potential}

\author[1]{Chengxiang Zhang\thanks{E-mail address: zcx@bnu.edu.cn}}
\author[2]{Luyu Zhang\thanks{E-mail address: zhangluyu@ncwu.edu.cn }}
\affil[1]{\footnotesize Laboratory of Mathematics and Complex Systems (Ministry of Education), School of Mathematical Sciences, Beijing Normal University, Beijing 100875,  China}
\affil[2]{\footnotesize School of Mathematics and Statistics, North China University of Water Resources and Electric Power, Zhengzhou 450046, China}

\date{}
\usepackage[left=3.1cm,right=3.1cm,top=3.2cm,bottom=3.2cm]{geometry}

\begin{document}
\maketitle
\begin{abstract}
	In this paper, we study the existence, uniqueness, nondegeneracy and some qualitative properties
	of positive solutions for  the logarithmic Schr\"odinger equations:
	\[ -\Delta u+ V(|x|)	u=u\log u^2,  \quad
			u\in H^1(\R^N).
   \]
   Here  $N\geq 2$ and $V\in C^2((0,+\infty))$ is allowed to be  singular at $0$ and repulsive at infinity (i.e., $V(r)\to-\infty$ as  ${r\to\infty}$).
   Under some general  assumptions, we show  the existence, uniqueness and nondegeneracy of this equation in the radial setting.
   Specifically,  
   these results apply to singular potentials such as $V(r)=\alpha_{1}\log r+\alpha_2 r^{\alpha_3}+\alpha_4$ with  $\alpha_1>1-N$, $\alpha_2, \alpha_3\geq 0$ and $\alpha_4\in\R$, 
   which is repulsive for $\alpha_1<0$ and $\alpha_2=0$. 
   We also investigate the connection between some power-law nonlinear Schr\"odinger 
   equation with a critical frequency potential 
   and the logarithmic-law Schr\"odinger 
   equation with $V(r)=\alpha\log r$, $\alpha>1-N$,  
   proving convergence of the unique positive radial solution from 
   the power type problem to the logarithmic type problem.
   Under a further assumption, we also derive the uniqueness and nondegeneracy results
   in $H^1(\R^N)$ by showing the radial symmetry of solutions.
\end{abstract}

	\vspace{6mm} \noindent{\bf Keywords:} Uniqueness and nondegeneracy; Logarithmic nonlinearity; Singular and repulsive potentials; Qualitative analysis
	
	\vspace{6mm}\noindent {\bf AMS} Subject Classification: 35A02, 35B25, 35B40,  35Q55.

\section{Introduction}

We investigate the existence,  uniqueness and nondegeneracy of positive solutions to
	 \begin{equation}\label{P1}
	 	 -\Delta u+ V(|x|)	u=u\log u^2,  \quad
	 		 u\in H^1(\R^N),
	 \end{equation}
where $N\geq 2$, $V\in C^2((0,+\infty))$ is allowed to be singular at $0$ and unbounded from below at infinity.
Problem \eqref{P1} comes from the study of standing waves to the nondispersive logarithmic Schr\"odinger equation
\begin{equation} \label{eq1.2}
 i\hbar\frac{\partial \Phi}{\partial t} = - \Delta \Phi + V(x) \Phi - \Phi\log|\Phi|^2 \quad \mbox{in}\quad \R^N\times \R, 
\end{equation}
where $\hbar$ is the Plank constant and $i$ is the imaginary unit.
For constant potential $V$, this equation was introduced in 
\cite{BM1,BM2}
to satisfy a specific tensorization property. See also   \cite{Carles} for a brief introduction on this property.
The equation \eqref{eq1.2} has been applied in  several
fields such as the Bose-Einstein condensation, nuclear physics,  
and the optics. See \cite{BM1,BM2,Ca,Carles,BCST,ZRZ,ZKg} and the reference therein for more discussion on \eqref{eq1.2}.
We also refer \cite{CG} for the dispersive case.

For the standing waves problem \eqref{eq1.2},
mathematical analysis has been carried out in recent years
on the existence and  multiplicity of solutions (\cite{dMS,Shuai,Jichao,ZW,ZZ,ITWZ,Sz,TZ}).
When $V$ is constant, up to a shift (see e.g. \cite{dMS,ITWZ,TZ}), \eqref{P1}
is equivalent to the following equation
\begin{equation}\label{econst}
	-\Delta u =u\log u^2,\quad u\in H^1(\R^N).
\end{equation}
It has been pointed in \cite{BM1,BM2,dMS,Troy} 
that \eqref{econst} has a positive solution, i.e., the Gausson. In \cite{dMS}, 
d'Avenia et al. also proved  the existence of infinitely many radially symmetric  sign-changing solutions to \eqref{econst} when $N\geq 3$.
Existence results are also derived for various a nonconstant potentials.
The authors in \cite{Shuai,Jichao} proved   existence of solutions for different type of potentials that are bounded from below.
In \cite{ITWZ,ZZ}, the authors consider the  semiclassical bound states of \eqref{P1} under some
general assumptions that allow the potential to be unbounded from below at a finite number of singular points or
at infinity.

Another interest in the study of nonlinear Schr\"odinger type equations is the uniqueness of the 
positive solution. 
We note that the study of  uniqueness for nonlinear Schr\"odinger equation with power type nonlinearity 
was started by Coffman \cite{Coffman} and
had been widely extended to general nonlinearities and  general 
potentials in \cite{Chenlin,Ko,Kwongzhang,McLeod1,McLeod2,Ni,Serrintang,KT,BO,Yanagida}.
Some of these results can be partially applied to the logarithmic equation \eqref{econst}.
The uniqueness result of  Serrin-Tang \cite{Serrintang} and the symmetry result of d'Avenia et al. \cite[Proof of Theorem 1.2]{dMS}
 imply that
the Gausson is the unique positive radial
solution   when $N\geq 3$. The special case $N=2$ was  treated by Troy \cite{Troy}.
When  $1\leq N\leq 9$, Troy \cite{Troy} proved that the Gausson is the only solution to \eqref{econst}, 
among positive  $u$'s satisfying $(u(r), u'(r))\to (0,0)$ as $r\to\infty$,
by energy estimates and Ricatti equation estimates.
However, it is not clear whether there is a unified way to derive existence and uniqueness results for problem 
\eqref{P1} with a general potential and for  $N\geq 2$.
 
In this paper, we will deal with  \eqref{P1} under some   
general assumptions on the  potential $V$ that is probably unbounded from below at $0$ or infinity.  A typical example in this case is the following problem appeared in 
\cite{ITWZ} as a limit equation of a semiclassical problem with singularities,
\begin{equation}\label{eq1.11'}
	-\Delta u+(\alpha \log |x|) u= u\log u^2,\quad u\in H^1(\R^N).
\end{equation}
It has been proved in \cite{ITWZ} that this equation has a positive ground state when $\alpha>0$. Here we assume $\alpha>1-N$, and note that the potential $\alpha\log |x|$ is 
repulsive when $\alpha\in(1-N,0)$.
We will obtain  general  existence, uniqueness and nondegeneracy  results
  for \eqref{P1}  which covers   \eqref{eq1.11'}. 
Moreover,  we also discuss the connection between  \eqref{eq1.11'} and a power nonlinear  equation with a critical frequency (\cite{BW1,BW2}),
which extends the result of 
\cite{WZ}.

  To state the results,
following the ideas of Byeon-Oshita \cite{BO} and  Kabeya-Tanaka \cite{KT}, for $\delta\geq0$, we consider the following perturbed problem which is slightly more general than \eqref{P1},
\begin{equation}\label{P1delta}
	-\Delta u+(V +\delta a ) u=(1+\delta b )u\log u^2,\quad u\in H^1(\R^N),
\end{equation}
where $a, b$ and $V$ satisfy
\begin{enumerate}
	\item[(Vab)] $a\in C^2_0(\R^N), b\in C^3_0(\R^N)$ 
	are radially symmetric functions, $b\geq 0$ and $b=1$ in a neighborhood of $0$.
	 \item[(V1)] $V\in C^2((0,+\infty))$, 
	 $ \liminf_{r\to+\infty} \frac{V(r)}{\log r}\in(1-N,+\infty]$,
	  $V(|x|)\in L_{loc}^q(\R^N)$ for some $q> N$.
\end{enumerate}
We have the following  result on existence.
\begin{Thm}\label{Thm1}
	Assume (Vab), (V1) and $\delta\geq 0$. Then \eqref{P1delta} has a positive radial solution 	
	$u\in C^{2-\gamma}(\R^N)\cap C^2(\R^N\setminus\{0\})$, where $\gamma:= \frac{N}q$.
	Moreover,
   for each $\tau\in(0,\frac12)$, there holds 
   \begin{equation}\label{decay}
	   \lim_{r\to\infty}u(r) e^{\tau r^2}=0. 
   \end{equation}
\end{Thm}
Assumption (V1) implies that repulsion would happen in the potential.
In fact, as $r\to+\infty$, $V(r)$ can approach the minus infinity with a logarimic strength.
To overcome the noncompactness caused by this repulsion,
 we will make use of the singular nature of logarithms  and 
the radial lemma of Strauss \cite{st} that for $u\in H^1_{r}(\R^N)$, 
\begin{equation}\label{decay1}
	u(r)=O(r^\frac{1-N}2)\quad \text{as}\quad r\to\infty.
\end{equation}
Here $H_r^1(\R^N)$ denotes the space of radial functions in $H^1(\R^N)$.
Then a positive solution can obtained to \eqref{P1delta} as a minimizer on the Nehari manifold in a suitable subspace of $H_r^1(\R^N)$.

To achieve uniqueness of this positive radial solutions in $H^1_r(\R^N)$,
it is nature to investigate every solution
satisfying the decay property \eqref{decay1}.  In fact, we will show the uniqueness under a weaker decay condition.
   To state the   problem, we  define
 $$\Theta:=\Set{\theta\in(1-N,+\infty)|\liminf_{r\to\infty}(V(r)-\theta \log r)>-\infty}.$$ 
 Note that under the assumption (V1), $\Theta$ is a nonempty subset of $\left(1-N,\liminf_{r\to+\infty} \frac{V(r)}{\log r} \right]$.
 For each $\theta\in \Theta$, we consider uniqueness of solutions to the following problem
\begin{equation} \label{P3}
	\left\{
		\begin{aligned}
			 &u''+\frac{N-1}{r}u'- (V+\delta a) u  +(1+\delta b) u\log u^2=0,\quad u>0\quad \mbox{in}\quad (0,+\infty),\\
			 &u'(0)=0,  \lim_{r\to\infty}r^{-\frac\theta2 } u(r)= 0.
		\end{aligned}
		\right.
	\end{equation}
By the singularity of the potential and nondecay assumption of $u$ in \eqref{P3} when $\theta>0$, we have to 
determine firstly the positivity of $u(0)$   and  a fast decay estimate for each solution $u$ of \eqref{P3}.
These will be done in Proposition \ref{LemA1}. In particular, every solution to \eqref{P3} is in $H_r^1(\R^N)$ and satisfies \eqref{P1delta}
by Proposition \ref{LemA1}.
 To achieve uniqueness, we  need a further assumption. Set
\[
	G(r) =  V(r)+\frac{(N-1)(N-3)}{4r^2}+(N-1)\log r. 
\]
and assume
 \begin{enumerate}
	\item[(V2)]  
       When $N=2$ or $3$, 
	    $G' >0$ in $(0,+\infty)$ and $\liminf_{r\to0^+} G'(r)>0$.
		When $N\geq 4$, $$\limsup_{r\to0^+}r^3 G'(r)<0
		$$
		and
		 $r^3G'(r)$ has a unique simple zero in $(0,+\infty)$.
\end{enumerate}
We have
\begin{Thm}\label{Thm1.2}
	Assume (Vab), (V1) and (V2). Then there is $\delta_0>0$   such that for every $\delta\in[0,\delta_0]$, problem \eqref{P3} admits a unique solution.
	 Especially, the positive radial solution to problem \eqref{P1delta} 	is unique in $H^1_r(\R^N)$.
\end{Thm}
Note that if $\liminf_{r\to+\infty} V(r)>-\infty$, we can take $\theta=0 \in \Theta$ in  problem \eqref{P3}.
Then we obtain a direct corollary of Theorem \ref{Thm1.2}.
\begin{Cor}\label{cor1.1}
	Assume (Vab), (V1), (V2) and $\liminf_{r\to+\infty} V(r)>-\infty$. Then there is $\delta_0>0$ such that for $\delta\in[0,\delta_0]$, the following problem has a unique  solution:
	\[ 
		\left\{
			\begin{aligned}
				 &u''+\frac{N-1}{r}u'- (V+\delta a) u  +(1+\delta b) u\log u^2=0,\quad u>0\quad \mbox{in}\quad (0,+\infty),\\
				 &u'(0)=0,  \lim_{r\to+\infty} u(r)= 0.
			\end{aligned}
			\right.
		\] 
\end{Cor}

\begin{Rek}
We give some more comments on Theorem \ref{Thm1.2} and Corollary \ref{cor1.1}.
	\begin{enumerate}
		\item[\rm (i)] For $\delta=0$,   the uniqueness of radial solutions to \eqref{P3} under (V1) and the following  weaker condition than (V2):

		when $N=2$ or $3$, $G' \geq 0$ in $(0,+\infty)$ ; and when $N \geq 4$, $\lim \sup _{r \to 0^+} r^3G'(r)<0$ and
		    the zero set of $r^3G'(r)$ in $(0, +\infty)$ is a  connected   nonempty set.
		\item[\rm(ii)] 
		 A sufficient condition for (V2) is 
		 the following:
		$V\in C^2((0,+\infty))$ satisfies $V' >  \frac{1-N}r$ for $N\geq 2$,  and  
		     additionally $rV'$ is   increasing   when $N\geq 4$.
		\item[\rm (iii)] We give typical examples of $V$ for  uniqueness of radial solutions to \eqref{P3} with $\delta=0$:
		$\alpha_{1}\log r+\alpha_2 r^{\alpha_3}+\alpha_4$ with  $\alpha_1>1-N$, $\alpha_2, \alpha_3\geq 0$ and $\alpha_4\in\R$.
		\item[\rm (iv)] Corollary \ref{cor1.1}  generalizes the uniqueness results of \cite{dMS} and \cite{Troy}  for  constant potentials and $\delta=0$. 
		Especially,  a problem  raised by Troy in \cite[Problem 1]{Troy} is resolved 
		by combining  Corollary \ref{cor1.1} and the symmetry result of \cite[Proof of Theorem 1.2]{dMS} that is true for $N\geq 2$.
		\item[\rm (v)] There is an  example  that multiple   positive radial solutions exist for repulsive potential:
		$V(r)=-\mu r^2$ for $\mu \in (0,\frac14)$. See \cite{CS} for the explicit expression of these solutions.
	\end{enumerate}

\end{Rek}
Now  we study uniqueness of positive solutions to problem \eqref{P1}.
To this end, we need to show any positive solution of \eqref{P1} is radially symmetric by assuming
\begin{enumerate}
	\item[(V3)] $V\in C((0,+\infty))$ is increasing and nonconstant.
\end{enumerate}
Note that (V3) implies $\liminf_{r\to+\infty} V(r)>-\infty$.
 We have
  \begin{Thm}\label{Thm1.3}
	  Assume (V3) and let $u\in C(\R^N)\cap C^2(\R^N\setminus\{0\})$ be a solution to
	  \begin{equation}
		  \label{P2}\left\{
	  \begin{aligned}
		  &- \Delta u+ V(|x|)	u=u\log u^2, \\
		  & u>0, u\to 0\ \mbox{as}\ \abs x\to \infty. 
	  \end{aligned}
	  \right.
	  \end{equation}
	  Then $u$ is radially symmetric about $0$. Moreover, $u'(r)<0$ for each $r>0$.
  \end{Thm}
  A direct corollary of Theorem \ref{Thm1.3} and Corollary \ref{cor1.1}   is
  the uniqueness of positive solutions to problem \eqref{P1}.
  \begin{Cor}
	Assume (V1), (V2) and (V3). Then problem \eqref{P1} $($or \eqref{P2}$)$ has a unique positive solution.
  \end{Cor}
  We also study the nondegeneracy of the unique radial solution. We say a positive radial solution
  $w$
  to \eqref{P1} is {\it nondegenerate} in $H^1_r(\R^N)$ (resp. $H^1(\R^N)$) if 
  the problem 
  \[
	  -\Delta \psi + V(|x|) \psi = (\log w^2 +2)\psi
  \]
  has no nontrivial solution in $H^1_r(\R^N)$ (resp. $H^1(\R^N)$).
  We have 
  \begin{Thm}\label{thm1.4}
	  Under the assumptions (V1) and (V2), the unique positive radial solution to \eqref{P1} is nondegenerate in $H^1_r(\R^N)$. Under the assumptions (V1) (V2) and (V3),
  the unique positive solution to \eqref{P1} is   nondegenerate in $H^1(\R^N)$.
  \end{Thm}
  For an equation with a $C^2$-smooth variational structure,
  the nondegeneracy of a solution can be  characterized   
 as  the nondegenerate Hessian  at the corresponding
 critical point of the variational functional. 
 This require a $C^2$-smooth variational functional associated to the problem.
 However, by the non-lipschitzian property of the logarithmic nonlinearity,
 there is no $C^2$-smooth variational structure for problem \eqref{P1}.
 We find a method in Section \ref{sec4.1} to recover smoothness so that we can 
apply the idea of \cite{BO,KT} to  show Theorem \ref{thm1.4}.

As we have discussed, problem  \eqref{eq1.11'} was first studied in \cite{ITWZ} as a  
limit equation for some semiclassical logarithmic Schr\"odinger equation.
In \cite{ITWZ}, it was also pointed that the semiclassical state behave like
 what had been studied in \cite{BW1,BW2} on the semiclassical power 
 nonlinear Schrödinger equations with a critical frequency.
So it seems that there should be a connection between the two type of semiclassical limit equations.
We will reveal  this connection by studying the limit profile of the following equation
 \begin{equation}\label{eq1.11}
	-\Delta u+|x|^{ {\alpha \sigma }} u=|u|^{2\sigma}u,  \quad u\in H^1(\R^N),
 \end{equation}
 where $\alpha>1-N$ and $\sigma\in(0,{2}/{(N-2)^+})$.
Here we denote that, throughout this paper,  
$d^+=\max\set{d,0}$ and $d^-=-\min\set{d,0}$ for $d\in \R$. Moreover, we take $1/0^+=+\infty$.
By \cite{BO},  we see that \eqref{eq1.11} has exactly one positive solution which is radially symmetric if $\alpha\geq 0$ and $\sigma\in (0,\frac{2}{(N-2)^+})$.
We will see that \eqref{eq1.11} has a unique positive radial solution 
even when $1-N<\alpha<0$  under some additional conditions.
Above all, for every $\alpha>1-N$, we can  give  the convergence result of the unique positive radial solution to \eqref{eq1.11} as $\sigma\to0^+$.
In fact, we have
\begin{Thm}\label{Thm1.5}
	Assume that $\alpha>1-N$ and $\sigma\in(0,{2}/{(N-2)^+})$. The following statements hold.
	 \begin{enumerate}
		 \item[\rm (i)]  If \begin{equation}\label{eq1.13}
			-\sigma\alpha<2\min\Set{1, N-1+\alpha },
		 \end{equation} 
		 then  \eqref{eq1.11} has a  nonnegative nontrivial   solution  $u_\sigma\in H_r^1(\R^N)$ satisfying $u_\sigma>0$ in $\R^N\setminus\{0\}$ and
		 \[
			u_\sigma(x)\leq C_\sigma \exp\big( -c_\sigma|x|^{\frac{\alpha\sigma+2}2} \big),\quad x\in \R^N,
	   \]
		where $C_\sigma, c_\sigma$ are positive constants independent of $x$. 
		
			\item[\rm (ii)] Assume $-\sigma\alpha<1$ in addition to \eqref{eq1.13}. Then   $u_\sigma\in C^{2-\gamma'}(\R^N)\cap C^2(\R^N\setminus\{0\})$ for 
			each $\gamma'\in(-\alpha\sigma,1)$ and $u_\sigma>0$ solves \eqref{eq1.11}
			 uniquely in $H_r^1(\R^N)$. 
			 \item[\rm (iii)]
			  For any $\gamma''\in(0,1)$, as $\sigma\to0^+$,
			 $\sigma^{\frac{\alpha}{4 }}u_\sigma( {\sigma}^{-\frac{1}{2 }}\cdot)$ converges to the unique positive radial solution of \eqref{eq1.11'}, in $H_r^1(\R^N)$ and $C^{2-\gamma''}(\R^N)\setminus C^2(\R^N\setminus B_{\gamma''}(0))$ . 
	 \end{enumerate} 
\end{Thm}

Note that assumptions \eqref{eq1.13} and $-\sigma\alpha<1$ hold automatically for either $\alpha\geq 0$ or $\sigma$ small enough.

The solution to \eqref{eq1.11} will be found in a suitable space and proved to belong to $H^1_r(\R^N)$. The proof of existence depends on a 
generalization of the radial lemma of Strauss \cite{st}. We remark that when  $\alpha=0$, Theorem \ref{Thm1.5}(iii) is in fact the convergence result  
considered in \cite[Theorem 1.1]{WZ}. We also refer \cite{ITWZ2,JLL} for related
works on sublinear elliptic equations and eigenvalue problems.

The remainder of the paper is organized as follows. 
In Section \ref{sec2}, we  introduce a suitable work space and show the existence of 
a radial  solution to \eqref{P1delta} and prove its positivity and the  Gaussian decay estimate
\eqref{decay}.
Section \ref{sec3} will be devoted to proving the uniqueness  and radial symmetry of positive  solutions to \eqref{P1}.
In addition, in Section \ref{sec4}, we also deal with the nondegeneracy of the unique radial solution to \eqref{P1} respectively in $H^1_r(\R^N)$ and in $H^1(\mathbb{R}^N)$.
In Section \ref{sec5}, by a revised radial lemma in a suitable function space and a decay estimate, 
we first obtain a positive radial solutions
to \eqref{eq5.1} and show the uniqueness of this solutions
in the radial  setting. 
Then we show that  the solution converges to unique
positive radial solution of \eqref{eq1.11'} as $\sigma\to0^+$, which extends the result of \cite{WZ}.
Finally, in the Appendix \ref{secA}, some technical results which are useful for the proof of uniqueness will be given. 
    
  \setcounter{equation}{0}
\section{Existence and decay estimate}\label{sec2}
Throughout the section, we assume (Vab) and (V1).
First, we show that the  problem 
\begin{equation}\label{P1B}
		-\Delta u+ V_\delta	u=B_\delta u\log u^2,  \quad u\in H^1(\R^N),
\end{equation}
	 admits a   radial ground state solution, where 
	 $V_\delta=V+\delta a $ and $B_\delta=1+\delta b $.
 Let 
\begin{equation}\label{alpha1}
	 \alpha_0:=\begin{cases} \displaystyle
	 	\half \left({ \liminf_{r\to+\infty}\frac{V(r)}{\log r}+1-N}\right)\quad &\mbox{if}\quad\displaystyle \liminf_{r\to+\infty}\frac{V(r)}{\log r}\leq 0,\\
	 	0\quad &\mbox{if}\quad \displaystyle\liminf_{r\to+\infty}\frac{V(r)}{\log r}>0.
	 \end{cases}
\end{equation}
Note that $\alpha_0\leq 0$ and $1-N<\alpha_0<\liminf_{r\to+\infty}\frac{V(r)}{\log r}$.
Set 
$$ E:=\Set{u\in H^1(\R^N) | \int_{\R^N} (V-\alpha_0\log |x|)^+ u^2<+\infty}
$$
with the norm
$$\|u\|=\left(\int_{\R^N}(|\nabla u|^2+u^2)+\int_{\R^N}(V-\alpha_0\log |x|)^+u^2\right)^\half.
$$
Note that 
\begin{equation}\label{va1}
	 \liminf_{r\to+\infty}\frac{V(r)-\alpha_0\log r}{\log r}=\liminf_{r\to+\infty}\frac{V(r)}{\log r}-\alpha_0>0.
\end{equation}
Then the embedding $E\subset L^p(\R^N)$ is compact for $p\in[2,2^*)$, where $2^*= {2N}/{(N-2)^+}$. 
As in \cite{ITWZ},  the Gagliardo--Nirenberg interpolation inequality is frequently used in this section.
For every $p\in(2,2^*)$, we have
\begin{equation}\label{eq:2.14}
\|u\|_{L^p(\R^N)} 
\leq C \norm{u}_{L^2(\R^N)}^{1-\nu_p}\|\nabla u\|_{{L^2(\R^N)}}^{\nu_p }
\quad \text{for all $ u \in H^1(\R^N)$}
\end{equation}
where $\nu_p =\frac{N(p-2)}{2p}\in (0,1)$ satisfies
$\nu_p \to 0$ as $p\to 2$.
     Moreover, by combining Young's inequality and \eqref{eq:2.14},
for any $\e>0$, there is a $C_\e>0$ such that
	\begin{equation}\label{eq:2.16}
		\begin{aligned}
			\| u \|_{L^p(\R^N)}^2 
			&\leq C_\e \| u \|_{L^2(\R^N)}^2
			+  \e \|\nabla u \|_{{L^2(\R^N)}}^2.
		\end{aligned}
	\end{equation}
	We consider another equivalent norm on $E$ for convenience.
\begin{Lem}\label{Lem21}
For each $\delta\geq 0$,	there is $\mu_\delta>0$ such that
	   the norm $\|\cdot\|_\delta$ defined by
	\begin{equation*} 
		\|u\|_\delta^2=\int_{\R^N}\left(|\nabla u|^2+(V_\delta-\alpha_0\log|x|) u^2+
		\mu_\delta  u^2\right),\quad u\in E
	\end{equation*}	
	is equivalent to $\|\cdot\|$ on $E$.
\end{Lem}

\begin{proof}
	It suffices to find $\mu_\delta>0$ such that for every $u\in E$,
	\begin{equation}\label{eq2.8}
		\int_{\R^N}\left(|\nabla u|^2+(V_\delta-\alpha_0\log|x|) u^2+
		\mu_\delta   u^2\right)  
		\geq\half\|u\|^2. 
	\end{equation}
In fact, by \eqref{va1} and (Vab),
there is $R_1>0$ such that $\supp (V_\delta-\alpha_0 \log |x|)^-\subset B_{R_1}(0)$.
Then by
the H\"older's inequality and \eqref{eq:2.16}, for some $C_1  > 0$, there holds
\[
	\begin{aligned}
	  \int_{\R^N}  (V_\delta-\alpha_0 \log |x|)^- u^2  
	\leq&  \|V_\delta-\alpha_0\log |x|\|_{L^q(B_{R_1}(0))} \| u \|_{L^{\frac{2q}{q-1}}(B_{R_1}(0))}^2\\
	\leq& \frac{1}{2 } \| \nabla u \|_{L^2(\R^N)}^2 + C_1 \| u \|_{L^2( \R^N )}^2
	\end{aligned}
\]
for all $u \in E$. So, setting $\mu_\delta=C_1+\half+\delta\max_{x\in\R^N}|a(x)|$, we have
\eqref{eq2.8}.	
\end{proof}

By   \cite[Lemma 2.1]{ITWZ},  there hold $uv\log u^2\in L^1(\R^N)$ for $u,v\in E$.
Especially, by \cite[Remark 2.1]{ITWZ},   we can define the  $C^1$-functional in $E$,
	$$I_\delta(u)=\frac12\int_{\R^N} |\nabla u|^2 + (V_\delta+B_\delta) u^2- \frac12 \int_{\R^N} B_\delta u^2\log u^2.$$
	Denote $E_r:= E\cap H^1_r(\R^N)$.
	Note that by the radial lemma of Strauss \cite{st} (see also \cite{BL}), there is $C_N>0$ independent of $u$,
such that 
\begin{equation}\label{radial}
	|u(x)|\leq C_N|x|^{\frac{1-N}2}\|u\|_{H^1(\R^N)}.
\end{equation}
We consider the minimization problem
\begin{equation}\label{cdelta}
	c_\delta=\inf_{u\in\mathcal N_r^\delta } I_\delta(u)
\end{equation}
where
	$$\mathcal N_r^\delta=\Set{u\in E_r\setminus\{0\} 
	|
	0=J_\delta(u):=I_\delta'(u)u=\int_{\R^N}|\nabla u|^2 +V_\delta u^2
	-\int_{\R^N}B_\delta u^2\log u^2}.
	$$
	
	\begin{Pro}\label{lemma2.1}
		The following statements hold for each $\delta \geq0$.
		\begin{enumerate}
			\item[\rm (i)] For each $u\in E_r\setminus\{0\}$, there is a unique $t_u\in\R$ such that $e^{t_u/2}u\in\mathcal  N_r^\delta\setminus\{0\}$. 
			 Moreover,
			$t_u= (\int_{\R^N}B_\delta u^2)^{-1} J_\delta(u)$, and it is the unique maximum point of the function $t(\in\R)\mapsto I_\delta (e^{t/2}u)$.
		
			\item[\rm (ii)]  $c_\delta>0$ and
			the minimization problem \eqref{cdelta} is attained by   
			  a  radial solution,  which is positive in $\R^N\setminus\{0\}$, to \eqref{P1B}.
		\end{enumerate}	
	\end{Pro}
	\begin{proof}
		1.  The first conclusion follows from  a direction calculation,  
		$$\frac{\d}{\d t} I_\delta (e^{\frac{t}2}u) =\half e^t\left(J_\delta(u)-t\int_{\R^N} B_\delta u^2\right).
		$$
		2.
We write
\begin{align*}
	I_\delta(u)=\frac12\int_{\R^N} |\nabla u|^2 + (V_\delta-\alpha_0\log|x|+\mu_\delta+ B_\delta) u^2- \frac12 \int_{\R^N}  (u^2\log (e^{\mu_\delta} |x|^{-\alpha_0}u^2)+\delta bu^2\log u^2).
\end{align*} 
Fix $p\in (2,2^*)$   such that
\begin{equation}\label{eq2.16}
	p-2+ \nu_p p<2.
\end{equation}
By \eqref{radial},
for some   $C_3>0$
\begin{equation}\label{eq30}
	\begin{aligned}
	 u^2\log (e^{\mu_\delta}& |x|^{-\alpha_0}u^2)+\delta bu^2\log u^2\\
	 \leq&  u^2\left(\log (e^{\mu_\delta} |x|^{-\alpha_0}u^2)\right)^++\delta bu^2\left( \log u^2\right)^+\\
		\leq &
		  u^2\left(\log (e^{\mu_\delta} u^2)\right)^+
		+ u^2\left(\log (e^{\mu_\delta} C_N \|u\|_{H^1(\R^N)}^{\frac{2\alpha_0}{1-N}}
		u^{2-\frac{2\alpha_0}{1-N}})\right)^++\delta bu^2\left( \log u^2\right)^+\\
		\leq& C_3 \left(|u|^p+\|u\|_{H^1(\R^N)}^{p-p_1} |u|^{p_1}\right)\quad \mbox{in $\R^N$},
	\end{aligned}
\end{equation}
where $p_1:=p-\frac{(p-2)\alpha_0}{1-N}\in(2,p]$.
 By \eqref{eq2.8} and \eqref{eq30}, we have 
$I_\delta(u)\geq \frac14\|u\|^2- C\|u\|^p
$ for some $C>0$.
Hence, there is $m_\delta>0$ such that 
$$I_\delta(u)\geq \frac18 m_\delta^2\quad\mbox{for}\quad  \|u\|=m_\delta\quad \mbox{and}\quad I_\delta(u)\geq 0\quad \mbox{for}\quad \|u\|\leq m_\delta.$$
Then 
$$c_\delta=\inf_{u\in \mathcal N_r^\delta } \sup_{t\in\R } I_\delta(e^{\frac{t}2}u)
\geq \frac18 m_\delta^2>0.
$$
To show that $c_\delta$ is achieved. We assume $u_n\in\mathcal N^\delta_r$ is such that 
$I_\delta(u_n)\to c_\delta>0.
$
Since  
\begin{equation}\label{eq2.24}
	\int_{\R^N} B_\delta u_n^2=2I_\delta(u_n)- J_\delta(u_n)=2I_\delta(u_n),
\end{equation}
we see that $\|u_n\|_{L^2(\R^N)}$ is bounded. By $J_\delta(u_n)=0$ and \eqref{eq30} we have
$$ \begin{aligned}
	\|u_n\|_\delta^2=&  \int_{\R^N}  (u_n^2\log (e^{\mu_\delta} |x|^{-\alpha_0}u_n^2)+\delta bu_n^2\log u_n^2)\\
	\leq& C_p \left(\|u_n\|_{L^p(\R^N)}^p+\|u_n\|_{H^1(\R^N)}^{p-p_1}\|u_n\|_{L^{p_1}(\R^N)}^{p_1}\right)\\
	\leq &C_p'\left(\|u_n\|_{L^2(\R^N)}^{(1-\nu_p)p}\|\nabla u_n\|_{L^2(\R^N)}^{\nu_p p}
	+\|u_n\|_{H^1(\R^N)}^{p-p_1}  
	\|u_n\|_{L^2(\R^N)}^{(1-\nu_{p_1})p_1}\|\nabla u_n\|_{L^2(\R^N)}^{\nu_{p_1}p_1}\right)\\
	\leq & C_p''(\|u_n\|_\delta^{\nu_p p}+\|u_n\|_\delta^{p-p_1+\nu_{p_1} p_1}).
\end{aligned}
$$
By \eqref{eq2.16},  we have $p-p_1+\nu_{p_1} p_1 < p-2+ \nu_pp<2$. Then $\|u_n\|_\delta$ is bounded.

Up to a subsequence, we assume $u_n\rightharpoonup u$ in $E$ and $u_n\to u$ in $L^2(\R^N)\cap L^p(\R^N)$.
By \eqref{eq2.24}, we obtain  $\int_{\R^N} B_\delta u^2=2c_\delta>0$ and $u\neq 0$.
By (i), there is $t_u$ such that $e^{t_u/2}u\in\mathcal N_r^\delta$.
To   prove the attainability of $c_\delta$,  it suffices to show that $t_u=0$.

Note that $\eqref{eq30}$ and  the boundedness of $\|u_n\|_\delta$  imply that 
$$u_n^2\left(\log (e^{\mu_\delta} |x|^{-\alpha_0}u_n^2)\right)^++\delta b u_n^2\left( \log u_n^2\right)^+\leq C (|u_n|^p+  |u_n|^{p_1})\quad \text{in}\quad \R^N.
$$
Then by the dominated convergence theorem,
up to a subsequence, we can conclude that, 
\begin{equation}\label{eq221}
	\int_{\R^N} u_n^2\left(\log (e^{\mu_\delta} |x|^{-\alpha_0}u_n^2)\right)^++\delta b u_n^2\left( \log u_n^2\right)^+\to \int_{\R^N} u^2\left(\log (e^{\mu_\delta} |x|^{-\alpha_0}u^2)\right)^++\delta bu^2\left( \log u^2\right)^+.
\end{equation}
Therefore, by (i), \eqref{eq221}, the weakly lower  semicontinuity of norm and  the Fatou's Lemma,
we obtain
\begin{align*}
	t_u\int_{\R^N} B_\delta u^2=  
	 J_\delta(u)
	\leq \liminf_{n\to+\infty} J_\delta(u_n)=0.
\end{align*}
On the other hand, by 
$$2c_\delta\leq 2I_\delta(e^{t_u/2}u)=e^{t_u}\int_{\R^N}B_\delta u^2,
$$
we have $t_u\geq 0$.  Then $t_u=0$ and $u\in \mathcal N_r^\delta$.

Since $c_\delta$ is also achieved by $|u|$. We may assume that $u\geq 0$ and solves
\eqref{P1B}. By the regularity theory, $u\in C^{2-\gamma}(\R^N)\cap C^2(\R^N\setminus\{0\})$.
Noting that $u$ satisfies
$$-\Delta u+V_\delta^+ u\geq B_\delta u\log u^2 \quad \mbox{in}\quad \R^N\setminus\{0\},
$$
by the maximum principle in \cite{V}, we have $u>0$ in $\R^N\setminus\{0\}$.
\end{proof}
The positivity and the decay estimate of the solution is obtained in the following.
	\begin{Pro}\label{LemA1}
		Let $u(r)$  solves \eqref{P3}. Then 
		\begin{enumerate}
			\item[\rm (i)] $u(0)>0$,
			\item[\rm (ii)]$\lim_{r\to+\infty}u(r) e^{\tau r^2}=0$ for any $\tau\in(0, 1/2)$.
		\end{enumerate}
		
  \end{Pro}
  \begin{proof}
	(i) Assume on the  contrary  that $u(0)=u'(0)=0$. We have
   $$(r^{N-1}u')'  -   r^{N-1}V_\delta(r)u + r^{N-1} B_\delta(r) u^2\log u^2=0,\quad r\in(0,+\infty),
   $$
   or 
   $$u(r)=\int_0^r s^{1-N}\d s \int_0^s \left(t^{N-1}V_\delta(t)u(t) - t^{N-1} B_\delta(t) u^2(t)\log u^2(t)\right)\d t.
   $$
   Set $v(r)=\max_{0\leq s\leq r}u(s)$. Then $0<u(r)\leq v(r)$ for $r >0$ and $v(r)$ is increasing with respect to $r$.
   Fix $r_0>0$   small enough such that $v(r_0)<e^{-1}$. Let $r_1\in(0,r_0]$ be such that $v(r_0)=u(r_1)$, by monotonicity,   
   we have
   \begin{align*}
	   v(r_0)= &u(r_1)
	   =\int_0^{r_1} s^{1-N}\d s \int_0^s \left(t^{N-1}V_\delta(t)u(t) - t^{N-1} B_\delta(t) u^2(t)\log u^2(t)\right)\d t\\
	   \leq& \int_0^{r_0} s^{1-N}\d s \int_0^s \left(t^{N-1}|V_\delta(t)|v(s) - t^{N-1} B_\delta(t) v^2(s)\log v^2(s)\right)\d t\\
	   \leq &C\int_0^{r_0} \left[s^{1-N} v(s) \left(\int_0^s t^{\frac{(N-1)^2}{N}\cdot\frac{N}{N-1}} \d t \right)^\frac{N-1}N\left(\int_0^s t^{N-1}|V_\delta(t)|^N
	   \d t \right)^\frac{1}N  -    v(s)\log v^2(s) \right]    \d s\\
	   \leq& C\int_0^{r_0} \left[s^{1-N} v(s) \left(\int_0^s t^{N-1}\d t\right)^\frac{N-1}N\|V(|\cdot|)\|_{L^N(B_{r_0}(0))}  -    v(s)\log v^2(s) \right]    \d s\\
	   \leq& C\int_0^{r_0} \left(v(s)-v(s)\log v^2(s) \right)\d s.
   \end{align*}
   Then by \cite[Lemma 1.4.1]{Ag}, we have $v\equiv 0$ in $(0,r_0)$. This is a contradiction.

	  (ii) Set $v(r)=r^{-\frac{\theta }2} u(r)$ and we have
	  $$v''+\frac{N-1+\theta}rv'+\frac{\theta}{2r^2}\left(\frac{\theta}2+N-2\right)v
	  -V_\delta(r) v +\theta B_\delta v\log r +B_\delta v\log v^2=0.
	  $$
	  Let $R_1>0$ be such that $V_\delta=V$ and $B_\delta=1$ for $r\geq R_1$.
	  Set 
	  $$\mu = -\inf_{r\geq R_1}\left(V(r)-\theta \log r-\frac{\theta}{2r^2}\left(\frac{\theta}2+N-2\right)\right).
	  $$
	  Then in $[R_1,+\infty)$, $v$ satisfies 
	  \begin{equation}\label{Av}
		   v''+\frac{N-1+\theta}r v'  +v\log (e^\mu  v^2)\geq 0.
	  \end{equation}
	  Fix $\tau_1\in(\tau,\frac12)$.
	  Since, $\lim_{r\to+\infty} v=0$, we can find  $R>R_1$ be such that for $r\geq R$,
	  $$(4\tau_1^2-2\tau_1)r^2-2\tau_1(N-1+\theta)-2\tau_1<0\quad \mbox{and}\quad 4v^2(r) e^\mu<e^{-2}.
	  $$
	  Now set 
	  $$ v_0(r)= 2v(R)e^{\tau_1 R^2}e^{-\tau_1 r^2}, \quad r\geq R.
	  $$
	  We have
	  \begin{equation}\label{Av0}
		   v_0''+\frac{N-1+\theta}rv_0'  +v_0\log (e^\mu v_0^2)\leq 0.
	  \end{equation}
	  We claim that 
	  $$v(r)\leq v_0(r)\quad \mbox{for}\quad r\geq R.
	  $$
	  Assume by contradiction that \[\sup_{r\geq R}\left(v(r)-v_0(r)\right)>0. \]
	  By $v(R)<v_0(R)$ and $\lim_{r\to+\infty}(v(r)-v_0(r))=0$, there is a local 
	  maximum point $R_2\in(R,+\infty)$ of $v-v_0$
	  such that 
	  $v(R_2)>v_0(R_2)$ and $v'(R_2)-v_0'(R_2)=0$.
	  By the strict decreasing property of $s\log s$ for $s\in (0,e^{-2})$, we have
	  \[v_0(R_2)\log (e^\mu v_0^2(R_2)) > v(R_2)\log (e^\mu v^2(R_2)).\]
	  Then by \eqref{Av} and \eqref{Av0}, $v''(R_2)-v_0''(R_2)>0$.
	  This is a contradiction. Then   as $r\to+\infty$,
	  \[  u(r)e^{\tau r^2}=r^{\frac\theta2}v(r)e^{\tau r^2}\leq r^{\frac\theta2}v_0(r)e^{\tau r^2}\to 0. \qedhere\]
	\end{proof}
\begin{Rek}\label{Rek2.1}
	Modifying the proof of Proposition \ref{LemA1} (i) slightly, 
	we can get the uniqueness of 
	the following initial value problem under assumption (Vab) and (V1):
\[ 
    \left\{
    \begin{aligned}
         &u''+\frac{N-1}{r}u'-V_\delta(r)u+B_\delta(r)u\log u^2=0,\\
         &u(0)=\beta,\quad
         u'(0)=0,\quad \beta\in\R.
    \end{aligned}\right. 
    \]
	In fact,  the assumption that $V(|x|)\in L^q(B_1(0))$  ensures the uniqueness of the solution near $0$.
	On the other hand, the uniqueness in $(0,+\infty)$ follows from the uniqueness criteria for initial value problems (see \cite[Theorem 3.5.1]{Ag}). 
\end{Rek}

 \begin{proof}[Completion of the proof of Theorem \ref{Thm1}]
	 Proposition \ref{lemma2.1}(ii) gives the existence of a radial state solution and Proposition \ref{LemA1} implies  this solution is 
	 in fact positive and has the desired decay property.
  \end{proof}

\setcounter{equation}{0}
\section{Uniqueness and symmetry}\label{sec3}
In this section, we show Theorem \ref{Thm1.2} and Theorem \ref{Thm1.3}. 
\subsection{Uniqueness in the radial setting}

Assuming (Vab), (V1) and (V2), to
  show the uniqueness of
\begin{equation} \label{4.1}
	\left\{
		\begin{aligned}
			 &u''+\frac{N-1}{r}u'- (V+\delta a) u  +(1+\delta b) u\log u^2=0,\quad u>0\quad \mbox{in}\quad (0,+\infty),\\
			 &u'(0)=0,  \lim_{r\to+\infty}r^{-\frac\theta2 } u(r)= 0,
		\end{aligned}
		\right.
	\end{equation}
we set $V_\delta =V+ \delta a$, $K_\delta =B_\delta^{-1}=(1+\delta b)^{-1}$ and
 $v=K_\delta^{-\frac14} r^{\frac{N-1}2} u$ in \eqref{4.1}. Then $v$ satisfies
 \begin{align*}
	K_\delta v''+\frac{K_\delta'}{2}v'-G_\delta v +   v\log v^2=0,
\end{align*}
where 
\begin{equation}\label{Gdelta}
	G_\delta = K_\delta V_\delta
	-\frac{K_\delta''}4+\frac{3(K_\delta')^2}{16K_\delta}+\frac{(N-1)(N-3)K_\delta}{4r^2}- \frac{\log K_\delta}{2}+(N-1)\log r.
\end{equation}
Let 
\begin{equation}\label{E}
	E_\delta(r;u)=\frac12 K_\delta (v')^2 -\frac12 G_\delta(r) v^2 + \frac12(v^2\log v^2 -v^2).
\end{equation}
Then
\begin{equation}\label{E'}
	E'_\delta(r)=-\frac12 G_\delta'(r)v^2.
\end{equation}
We have
\begin{Lem}\label{lem3.1}
	Assume (Vab), (V1), (V2)  and let $u$ be a solution to \eqref{4.1}. Then 
there is $\delta_0>0$ independent of $u$, such that for $\delta\in[0,\delta_0]$,
 $E_\delta(r;u)>0$ if  $r> 0$ and $\lim_{r\to+\infty} E_\delta(r;u)=0$.
 	
\end{Lem}
\begin{proof}
	We write $E_\delta(r)=E_\delta(r;u)$ for brevity.
	By (V1),  there is $r_n\in [1/n, 2/n]$ for each $n$, such that 
$$ \begin{aligned}
	r_n^{N-1}V(r_n)=n\int^{\frac2n}_{\frac1n} r^{N-1}V(r)\d r\leq 
	n\left(\int^{\frac2n}_\frac1{n} r^{N-1}\d r\right)^{1-\frac1N}\left(\int_{\frac1n}^\frac2n r^{N-1}|V(r)|^N\d r\right)^\frac1N\\
	=\left(\frac{2^N-1}N\right)^{1-\frac1N} n^{2-N}\left(\int_{\frac1n}^\frac2n r^{N-1}|V(r)|^N\d r\right)^\frac1N\to
	0\quad \mbox{as}\quad n\to+\infty.
\end{aligned}
$$
Hence $\liminf_{r\to0^+} r^{N-1} V(r)\leq 0$.
Then by \eqref{Gdelta}, (Vab), Proposition \ref{LemA1} (ii) and Lemma \ref{Cor2.1}, we have
\begin{equation}\label{Gv}
	\liminf_{r\to 0^+}G_\delta(r)v^2\leq \begin{cases}\displaystyle
	   -\infty, &\text{if}\ N=2,\cr
	   0, &\text{if}\ N\geq 3.
	   \end{cases}
	   \quad\mbox{and}\quad \liminf_{r\to +\infty}G_\delta(r)v^2=0\quad \mbox{if}\  N\geq 2.
\end{equation}
On the other hand, by
 \[
	\label{E2} v'(r)=-\frac14 K_\delta^{-\frac54} K_\delta' r^{\frac{N-1}2}u+ \frac{N-1} 2K_\delta^{-\frac14} r^{\frac{N-3}2} u    +K_\delta^{-\frac14} r^{\frac{N-1}2} u',
 \] 
 \eqref{Gdelta}, (Vab), Proposition \ref{LemA1} (ii) and Lemma \ref{Cor2.1}, we can get
 \begin{equation}
	\label{E3}\lim_{r\to0^+} v'(r)=\begin{cases}
		+\infty, &\text{if}\ N=2,\crcr
		(1+\delta)^\frac14u(0), &\text{if } N=3,\crcr
		0,&\text{if}\  N\geq 4.\crcr
		\end{cases}
		\quad \mbox{and}\quad \lim_{r\to+\infty} v'(r)=0.
 \end{equation}
 By \eqref{E}, \eqref{Gv} and \eqref{E3}, we get
\begin{equation}
	\label{E0}\limsup_{r\to0^+}E_\delta(r)\geq \begin{cases}\displaystyle
        +\infty, &\text{if}\ N=2,\cr
		\half (1+\delta)^{-\frac12} (u(0))^2, &\text{if}\ N=3,\cr
		0, &\text{if}\ N\geq 4.
		\end{cases}
\end{equation} 
and
\begin{equation}
	\label{E1}\limsup_{r\to+\infty}E_\delta(r)=0.
\end{equation} 

\noindent
{\bf Case $N=2$ or $3$}

In this case, we show that there is $\delta>0$ such that for $\delta\in[0,\delta_0]$,
 $E_\delta$ is strictly decreasing in $(0,+\infty)$. 
 By \eqref{E0} and \eqref{E1}, it suffices to show that $E_\delta'(r)\neq 0$ for each $r>0$.
 Arguing indirectly,
assume  there are $\delta_n\to 0$ and $r_n>0$ such that $E_{\delta_n}'(r_n)=0$.
By \eqref{E'}, we have $G_{\delta_n}'(r_n)=0$. By (Vab) and (V2), $r_n$ is bounded.
Otherwise, along a subsequence, $G'(r_n)=G_{\delta_n}'(r_n)=0$. This is a contradiction.
On the other hand, by (Vab) and (V2),  there is $\rho_1>0$ independent of $n$ such that
$G_{\delta_n}'(r)>0$ for  $r\in(0,\rho_1)$.
Then we may assume that $r_n \rightarrow r_{0} \geq \rho_1>0$. This implies $G^{\prime}(r_0)=0$ and it is a contradiction. 
 
\noindent
{\bf Case $N\geq 4$}

Let   $r_0>0$ be the unique simple zero of $r^3 G'(r)$. 
By $\limsup_{r\to0^+}r^3 G'(r)<0$,
 we have $r^3 G'(r)<0$   in $(0, r_0)$, $r^3 G'(r)>0$   in $(r_0, \infty)$ and $(r^3 G'(r))'|_{r=r_0}>0$. 
 Note also that $r^3G_\delta'(r)=r^3G'(r)$ for large $r$.
 Then, from (Vab) and (V2),  there is $\delta_0>0$
  such that 
 for each $\delta\in[0,\delta_0]$,
 $r^3G_\delta'(r)$ has a unique simple zero $r_\delta$ satisfying $\lim _{\delta \rightarrow 0} r_{\delta}=r_0$,
 $G_\delta'(r)<0$  for $0<r<r_\delta$ and $G_\delta'(r)>0$ {for} $r>r_\delta$.
 Hence $E_{\delta}$ is strictly increasing in $\left(0, r_{\delta}\right)$ and
strictly decreasing in $\left(r_{\delta}, \infty\right)$.
This proves the lemma for $N\geq 4$.

We remark that in  either case ($N=2,3$ or $N=4$), the fact that $E_\delta(r)$ is strictly decreasing
for large $r$ implies that $\lim_{r\to+\infty}E_\delta(r)=0$.
\end{proof}
Now we are ready to prove the uniqueness of \eqref{4.1}.
\begin{proof}[Proof of Theorem \ref{Thm1.2}]
Assume by contradiction that for some $\delta\in[0,\delta_0]$, 
 problem \eqref{4.1} has two solutions $u_1$, $u_2$
such that $0<u_1(0)<u_2(0)$. By Lemma \ref{prop4.1},
we may assume that $u_2$ intersects $u_1$ only once. Then  by Lemma \ref{lem4.2},
we can get, in $(0,+\infty)$,
\begin{align*}
&\frac{\d}{\d r}\left(\frac{u_1}{u_2}\right)>0,\quad \text{or equivalently},\quad
\frac{\d}{\d r}\left(\frac{v_2}{v_1}\right)^2=
\frac{\d}{\d r}\left(\frac{u_2}{u_1}\right)^2<0,
\end{align*}
where $v_i=K_\delta^{-\frac14} r^{\frac{N-1}2} u_i$, $i=1,2$.
We can also see
 \begin{equation}
	\label{A17}
	0<\left(\frac{v_2(r)}{v_1(r)}\right)^2=\left(\frac{u_2(r)}{u_1(r)}\right)^2 <\left(\frac{u_2(0)}{u_1(0)}\right)^2\quad\mbox{for}\quad r>0. 
 \end{equation}
 Setting $E_{i}(r)=E_\delta(r; u_i)$, $i=1,2$,
it follows from \eqref{E} and \eqref{E'} that
\begin{align}\label{eq3.13}
\frac{\d}{\d r}\left(\left(\frac{v_2}{v_1}\right)^2E_{1}-E_{2}\right)
= \left(\frac{v_2}{v_1}\right)^2\frac{\d E_{1}}{\d r}-\frac{\d E_{2}}{\d r}
+E_{1}\frac{\d}{\d r}\left(\frac{v_2}{v_1}\right)^2
=E_{1}\frac{\d}{\d r}\left(\frac{u_2}{u_1}\right)^2 <0,
\end{align}
and
\begin{align*}
\left(\frac{v_2}{v_1}\right)^2E_{1}-E_{2}
=&\left(\frac{v_2}{v_1}\right)^2\left(\half K_\delta (v_1')^2-\half G_\delta(r)v_1^2+\half v_1^2(\log v_1^2-1)\right)\\
&\qquad \qquad-\left(\half K_\delta (v_2')^2-\half G_\delta(r)v_2^2+\half v_2^2(\log v_2^2-1)\right)\\
=&\half\left(K_\delta\left(\frac{v_2}{v_1}\right)^2(v_1')^2-K_\delta(v_2')^2+v_2^2\log v_1^2-v_2^2\log v_2^2\right).
\end{align*}
Let $(0, r_1)$ be an nonempty interval where $K_\delta=(1+\delta)^{-1}$.
By \eqref{E2}, we can check,
\begin{align*}
\left(\frac{v_2}{v_1}\right)^2(v_1')^2-(v_2')^2
=K_\delta^{-\frac12}r^{N-2}\left(\frac{u_2 u_1'}{u_1}-u_2'\right)\left((N-1)u_2+\frac{ru_2 u_1'}{u_1}+r u_2'\right)\to 0, \quad\mbox{as}\ r\to0.
\end{align*}
As a result, 
\begin{equation}\label{eq3.14}
	\lim_{r\to 0^+}\left(\left(\frac{v_2}{v_1}\right)^2E_{1}-E_{2}\right)=
	\half v_2(0)^2\log\left(\frac{u_1(0)}{u_2(0)}\right)^2 
	 =0,
\end{equation}
By Lemma \ref{lem3.1} and \eqref{A17}
we obtain that
\begin{equation}\label{eq3.15}
	\lim_{r\to+\infty}\left(\left(\frac{v_2(r)}{v_1(r)}\right)^2E_1(r)-E_2(r)\right)=0.
\end{equation}
Then we arrive at a contradiction by \eqref{eq3.13}, \eqref{eq3.14} and \eqref{eq3.15}.
\end{proof}
\subsection{Radial symmetry}

 Assume (V3) and
	let $u\in C(\R^N)\cap C^2(\R^N\setminus\{0\})$ solves
\eqref{P2}.
By means of moving plane method (see \cite{GNN}),
we show that $u$ is radially symmetric about $0$ and $u'(r)<0$ for each $r>0$.
\begin{proof}[Proof of Theorem \ref{Thm1.3}]

Denote $x=(x_1,x')\in\R\times\R^{N-1}$, and   for  $\lambda\in\R$,
 set 
$$\Sigma_\lambda=\set{x\in\R^N|x_1<\lambda},
\quad x_\lambda=(2\lambda-x_1,x'),\quad u_\lambda(x)=u(x_\lambda)\quad\mbox{and}\quad U_\lambda=u_\lambda-u.
$$
Then in $\Sigma_{\lambda}$, we have
\begin{equation}
	\label{wl}-\Delta U_\lambda+V(|x|) U_\lambda
	=(V(|x|)-V(|x_\lambda|)) u_\lambda
	+u_\lambda\log u_\lambda^2-u\log u^2. 
\end{equation}

\noindent
\textbf{Step 1.} 
Take 
$R>1$
 such that $u(x)< \min\{e^{-1},u(0)\}$ if $|x|\geq R$.
 We show that 
 $U_\lambda \geq 0$ in $\Sigma_\lambda\setminus B_R(0)$ for each $\lambda\leq0$.

Otherwise, since $U_\lambda(x)\to 0$ as $\abs x\to+\infty$ and $U_\lambda|_{\partial \Sigma_\lambda}=0$, 
we assume $U_\lambda$ reaches its negative minimum at some $\hat x\in\Sigma_\lambda\setminus B_R(0)$. We note that
$\hat x,\hat x_\lambda\neq 0$ by the choice of $R$. 
So $U_\lambda$ is $C^2$ near $\hat x$.
Noting that $s\log s^2$ is strictly decreasing in $(0,e^{-1})$, 
 we have $u_\lambda\log u_\lambda^2-u\log u^2>0$ at $\hat x$. 
 Since we can assume without loss of generality that $V(r)\geq 1$ for $r\geq 1$,
we have $$\Delta U_\lambda(\hat x)<V(|\hat x|) U_\lambda(\hat x)
+(V(|\hat x_\lambda|)-V(|\hat x|)) u_\lambda(\hat x)\leq0.$$ 
This is a contradiction since $\hat x$ is the minimum point of $U_\lambda$.

\noindent
\textbf{Step 2.}
Set $\lambda_0=\sup\set{\lambda<0 | U_{\lambda'}\geq 0\ \text{in}\ \Sigma_{\lambda'} \ \text{for any}\ \lambda'\in(-\infty,\lambda]}$.
Step 1 implies that 
$U_{\lambda}\geq 0$ in $\Sigma_\lambda$ for each $\lambda\leq -R$ and hence $\lambda_0\geq -R$.
 We   claim that 
$U_{\lambda}>0$ in $\Sigma_{\lambda}$ for $\lambda\leq\lambda_0$.
In fact,   by \eqref{wl}, there holds
\begin{equation}\label{w}
	\begin{aligned}
		-\Delta U_\lambda +  U_\lambda(V(|x|) -\log U_\lambda^2)=(V(|x|)-V(|x_\lambda|))u_\lambda +\int_{u}^{u_\lambda}(\log s^2-\log(s-u)^2)\d s\geq 0, 
	\end{aligned}
\end{equation}
where $V(|x|) -\log U_\lambda^2$ is bounded from below in $\Sigma_\lambda$. By maximum principle (\cite{CL,GT}), 
either $U_{\lambda}\equiv0$ or 
$U_{\lambda}>0$ in $\Sigma_{\lambda}$. 
Since $V'(r)\not\equiv0$, we have $U_\lambda >0$.

\noindent
\textbf{Step 3.}  We prove $\lambda_0=0$.
Assume by contradiction that $\lambda_0<0$,
we prove  that there exists $\delta_0>0$ such that for any $\delta\in(0,\delta_0]$
$$U_{\lambda_0+\delta}\geq0 \quad \text{in}\quad \Sigma_{\lambda_0+\delta}.
$$
Arguing by contradiction, for $\e_n\to 0^+$, assume that $x^n\in \Sigma_{\lambda_0+\e_n}$ attains the negative minimum of $U_{\lambda_0+\e_n}$.
  We note that by Step 1, $\abs{x^n}\leq R$ for all $i$. 
We assume along a subsequence, $x^n\to x^0$.
Then
$$U_{\lambda_0}(x^0)\leq0,\quad \nabla U_{\lambda_0}(x^0)=0,
$$
 which implies $x^0\in\partial\Sigma_{\lambda_0}$.
 By \eqref{w} and Hopf Lemma (\cite{CL,GT}), we get a contradiction
 $$\dfrac{\partial U_{\lambda_0}(x^0)}{\partial x_1}<0.
 $$

 Now we have shown that $u_\lambda \leq u$ and $\frac{\partial u}{\partial x_1}>0$ in 
 $\Sigma_0$ by Step 3. 
 Then we can complete the proof since similar arguments hold for any direction in $\R^N$.
\end{proof}
\setcounter{equation}{0}
\section{Nondegeneracy}\label{sec4}
Assume (V1), (V2).
Let $w>0$ be the unique radial positive solution to \eqref{P1}.
If $\psi\in H^1(\R^N)$  weakly solves
\begin{equation}\label{eqpsi}
	-\Delta \psi+V(|x|) \psi=(\log w^2+2)\psi,
\end{equation}
then  we have
\begin{align*}
 	\int_{\R^N}|\nabla \psi|^2+\int_{\R^N}\left((V-\alpha_0\log |x|+\mu_0)+(\log |x|^{-\alpha_0}w^2)^-\right)\psi^2 
	= \int_{\R^N}  ((\log |x|^{-\alpha_0}w^2)^+ +2+\mu_0) \psi^2,
\end{align*}
where $\mu_0>0$ is the constant determined in Lemma \ref{Lem21} for $\delta=0$.
Noting that $|x|^{-\alpha_0}w\to 0$ as $|x|\to+\infty$ by Proposition \ref{LemA1} (ii), we have $ \int_{\R^N}  ((\log |x|^{-\alpha_0}w^2)^+<+\infty.$
Then by Lemma \ref{Lem21},
we can conclude that $\psi\in E$.
Choosing $R>0$ such that $|x|^{-\alpha_0}w^2 <w<1$ when $|x|\geq R$. 
Then,
$$\int_{\R^N} (\log w)^-\psi^2=-\int_{\R^N\setminus B_R(0)} (\log w)\psi^2
	<-\int_{\R^N\setminus B_R(0)}  (\log |x|^{-\alpha_0}w^2)\psi^2=\int_{\R^N}  (\log |x|^{-\alpha_0}w^2)^-\psi^2.$$
Therefore, $\int_{\R^N} (\log w)^-\psi^2<+\infty$. By $ \int_{\R^N}  ((\log |x|^{-\alpha_0}w^2)^+<+\infty$ and  \eqref{eqpsi}
\begin{equation}\label{eq41}
\int_{\R^N} |\log w|\psi^2<+\infty\quad \mbox{and}\quad \int_{\R^N}|V|\psi^2<+\infty.
\end{equation}
	
To achieve the nondegeneracy, we first show that $\psi=0$ if $\psi\in E_r$ and next show 
$\psi\in E_r$ if we further assume   (V3).
\subsection{Nondegeneracy in the radial setting}\label{sec4.1}

\begin{proof}[Proof of $\psi=0$ if $\psi\in E_r$]
	First note that
	\begin{equation}\label{eq4.1}
		 0=\int_{\R^N} (\nabla \psi \nabla w +V\psi w -\psi w\log w^2 -2\psi w)=-2\int_{\R^N}\psi w.
	\end{equation}
	

	Assume $\psi\in E_r\setminus\{0\}$.
	We can  choose a nonnegative radial function $b\in C_0^\infty(\R^N)$ such that $b=1$ in $B_{r_1}(0)$ in a neighborhood of $0$
	such that 
	$$\int_{\R^N} b\psi^2>0.
	$$
	Set $a= b\log w^2$.
	By Theorem \ref{Thm1},
	 there is $\delta>0$  such that $w$ is the unique solution to
	\begin{equation}\label{eq411}
		-\Delta u + V_\delta(|x|) u=B_\delta(|x|) u\log u^2,
	\end{equation}
	where $V_\delta =V +\delta a=V+\delta b\log w^2$ and  $B_\delta =1+\delta b$.
We may also assume that $\delta$ is chosen small enough such that
\begin{equation}\label{eq4.2}
	\int_{\R^N} b \psi^2-\frac{\delta(\int_{\R^N}bw\psi)^2}{\int_{\R^N}B_\delta w^2}>0.
\end{equation}
By Proposition \ref{lemma2.1},
\begin{equation}\label{eq4.3}
	2I_\delta(w)=2\inf_{u\in\mathcal N_r^\delta} I_\delta(u)=\inf_{u\in\mathcal N_r^\delta} (2I_\delta(u)- J_\delta(u))= \inf_{u\in\mathcal N_r^\delta} \int_{\R^N}B_\delta u^2=\int_{\R^N} B_\delta w^2.
\end{equation}
Take a radial function 
$\eta\in C^\infty_0(B_{1}(0);[0,1])$   such that $\eta=1$ in  
	$B_{1/2}(0)$ and $|\nabla \eta|\leq 4$.
	For each  $n\geq 1$, we denote $\psi_n=\eta(n^{-1}{\cdot})\psi$.
	Noting that $\psi\in  C^1(\R^N)\cap C^2(\R^N\setminus\{0\})$ by regularity, we can find $s_n>0$ such that $w+s\psi_n>0$ for each $s\in (-s_n,s_n)$.
	Then the map $s\mapsto \int_{\R^N}(w+s \psi_n)^2\log (w+s\psi_n)^2$ is  $C^2$-continuous   
	for $s\in (-s_n,s_n)$.
For each $s\in (-s_n,s_n)$, denote  
$$  P_n(s):=J_\delta(w+s\psi_n),\quad Q_n(s):=\int_{\R^N} B_\delta(w+s\psi_n)^2  \quad \text{and}\quad t_n(s):= \frac{P_n(s)}{Q_n(s)}.
$$
Then by Proposition \ref{lemma2.1} (i),  we have 
 $$e^{\frac{t_n(s)}2}(w+s\psi_n)\in \mathcal N_r^\delta\quad \mbox{for}\quad s\in (-s_n,s_n).$$
 Let us check by \eqref{eqpsi},\eqref{eq41},\eqref{eq4.1} and \eqref{eq411},
\begin{equation}\label{eq4.4}
	\left.\begin{aligned}
		P_n'(0)&=2\int_{\R^N}(\nabla w\nabla \psi_n +V_\delta w\psi_n-B_\delta w\psi_n\log w^2-B_\delta w\psi_n)
		\to-2 \delta\int_{\R^N} b w\psi,\\
		 P_n''(0)&=2\int_{\R^N}( |\nabla \psi_n|^2 +V_\delta \psi_n^2-B_\delta \psi_n^2\log w^2-3B_\delta \psi_n^2)\to
		 -4\delta\int_{\R^N} b \psi^2-2\int_{\R^N}B_\delta \psi^2,\\
		Q_n'(0)&=2\int_{\R^N}B_\delta w\psi_n\to 2\int_{\R^N}\delta b w\psi,
		 \quad  Q_n''(0)\to 2\int_{\R^N}B_\delta  \psi^2,
		\end{aligned}\right\}
\end{equation}
 and   
\begin{equation}\label{eq4.5}
	 t_n(0)=0,\quad t_n'(0)Q_n(0)=  P_n'(0),
	  \quad t_n''(0)Q_n(0)+2t_n'(0)Q_n'(0)=P_n''(0).
\end{equation}
 Then we have by \eqref{eq4.4},\eqref{eq4.5} and \eqref{eq4.2},
 \begin{align*}
	 \left.\frac{\d^2}{\d s^2}\right|_{s=0}\left(e^{t_n(s)}Q_n(s)\right)&=
	 t_n''(0) Q_n(0) 
	+2t_n'(0)Q_n'(0)+(t_n'(0))^2 Q_n(s) + Q_n''(0)\\
	&=P_n''(0)+\frac{(P_n'(0))^2}{Q_n(0)}+Q_n''(0)\\
	&\to -4\delta \left(\int_{\R^N} b \psi^2-\frac{\delta(\int_{\R^N}bw\psi)^2}{\int_{\R^N}B_\delta w^2}\right)<0.
 \end{align*}
 This is a contradiction
because the function 
$e^{t_n(s)}Q_n(s)$, $s\in(-s_n,s_n)$ attains its minimum at $s=0$ by \eqref{eq4.3}.
\end{proof}

\subsection{Nondegeneracy in {$H^1(\R^N)$} }
In this subsection, we assume (V1), (V2) and  (V3).
It suffices to show the following result to obtain $\psi=0$ in $H^1(\R^N)$.
\begin{Lem} Let  $w$ be the unique solution to 
\eqref{P1}.
And $\psi\in H^1(\R^N)$ satisfies 
$$-\Delta \psi+V(|x|) \psi=(\log w^2+2)\psi.
$$
Then $\psi(x)=\psi(\abs x)$.

\end{Lem}
\begin{proof}
The proof is similar to \cite[Lemma A.4]{BO} but more subtle due to lack of regularity.
If $\psi$ is not radially symmetric, we assume without loss of generality that
$$\phi(x):=\psi(x_1,x_2,\cdots,x_N)-\psi(-x_1,x_2,\cdots,x_N)\not\equiv 0,
$$where $x=(x_1,x_2,\cdots,x_N)\in\R^N$. 
We note that
\begin{equation}\label{A19}
	 -\Delta \phi+V(|x|)\phi=(\log w^2+2)\phi.
\end{equation}
So by the regularity theory, 
$\phi\in  C^1(\R^N)\cap C^2(\R^N\setminus\{0\})$.
Let $\Omega$ be a connected component of 
$\set{x\in\R^N | \phi(x)>0}$.
 Since $\phi(x)=0$ when $x_1=0$.
We may assume $$\Omega\subset \set{x=(x_1,x_2,\cdots,x_N)|\ x_1>0}.$$
For $\e>0$, set $\Omega_{\e}:=\set{x \in \Omega \mid \phi(x)>\e}$, which is a bounded subset
of $\Omega$.
 By Sard's theorem, there exists $\e_{m}>0$ with $\lim _{m \rightarrow \infty} \e_{m}=0$ such that $\left\{\e_{m}\right\}_{m=1}^{\infty}$ are regular values of $\phi .$ Note that
\begin{equation}\label{A20}
	-\Delta \frac{\partial w}{\partial x_{1}}+V(|x|) \frac{\partial w}{\partial x_{1}}+V'(|x|)\frac{x_1}{|x|} w= (\log w^2+2)\frac{\partial w}{\partial x_{1}}
\end{equation}
By \eqref{A19} and \eqref{A20}, we have
$$\phi\Delta \frac{\partial w}{\partial x_{1}}-\frac{\partial w}{\partial x_{1}}\Delta \phi
-V'(|x|)\frac{x_1}{|x|} w\phi=0.
$$
 Integrating by parts on $\Omega_{\e_{m}}$, we obtain
\begin{equation}\label{A12}
	 \int_{\partial \Omega_{\e_{m}}} \frac{\partial^{2} w}{\partial x_{1} \partial \nu} \phi-\int_{\Omega_{\e_{m}}} V'(|x|)\frac{x_1}{|x|} w\phi= \int_{\partial \Omega_{\e_{m}}} \frac{\partial \phi}{\partial \nu} \frac{\partial w}{\partial x_{1}},
\end{equation}
 where $\nu$ denotes the outward unit vector normal to $\partial \Omega_{\e_{m}}.$  
By $\supp \left(V(|x|)-\log w^2-2\right)^-\subset B_R(0)$ for some $R>0$ and $\frac{\partial w}{\partial x_1}<0$ in $\set{x|x_1>0}$,
 we can check that 
\begin{equation}\label{A13}
	\begin{aligned}
		 &\int_{\partial \Omega_{\e_{m}}} \frac{\partial^{2} w}{\partial x_{1} \partial \nu} \phi
		-\int_{\Omega_{\e_{m}}} V'(|x|)\frac{x_1}{|x|} w\phi\\
		=&\e_m \int_{  \Omega_{\e_{m}}}\Delta \frac{\partial w}{\partial x_{1}}-\int_{\Omega_{\e_{m}}} V'(|x|)\frac{x_1}{|x|} w\phi\\
		=&\e_m \int_{\Omega_{\e_{m}}}\left(V(|x|)-\log w^2-2\right)\frac{\partial w}{\partial x_{1}}-\int_{\Omega_{\e_{m}}}V'(|x|)\frac{x_1}{|x|}(\phi-\e_m) w\\
		\leq&\e_m \int_{B_R(0)}\left(V(|x|)-\log w^2-2\right)^-\left|\frac{\partial w}{\partial x_{1}}\right|-\int_{\Omega_{\e_{m}}}V'(|x|)\frac{x_1}{|x|}(\phi-\e_m) w\\
		=&C\e_m-\int_{\Omega_{\e_{m}}}V'(|x|)\frac{x_1}{|x|}(\phi-\e_m) w,
	\end{aligned}
\end{equation}
where $C$ is a constant independent of $m$.
Therefore, by $V'(r)\geq 0$, we have 
$$\limsup_{m\to\infty}\left(\int_{\partial \Omega_{\e_{m}}} \frac{\partial^{2} w}{\partial x_{1} \partial \nu} \phi
-\int_{\Omega_{\e_{m}}} V'(|x|)\frac{x_1}{|x|} w\phi\right)\leq 0.
$$
 We next claim that 
 $\partial \Omega=\set{x | x_{1}=0}.$
 Otherwise, 
  $\partial \Omega \cap \set{x|x_1>0}\neq\emptyset$. On $\partial \Omega \cap \set{x|x_1>0}$, we have 
 $\frac{\partial w}{\partial x_{1}}<0$ and, by the Hopf lemma,
 $\frac{\partial \phi}{\partial \nu}<0$.
 We can check  that 
 $$\liminf_{m\to\infty}\int_{\partial \Omega_{\e_{m}}} \frac{\partial \phi}{\partial \nu} \frac{\partial w}{\partial x_{1}}\geq \int_{\partial \Omega \cap \set{x|x_1>0}} \frac{\partial \phi}{\partial \nu} \frac{\partial w}{\partial x_{1}}>0,
 $$
 which is a contradiction.
 Then $\partial \Omega=\set{x | x_{1}=0}$, and thus $\Omega=\set{x|x_1>0}$.
 By \eqref{A12} and \eqref{A13},
 we have
 $$\limsup_{m\to\infty}\int_{\Omega_{\e_{m}}}V'(|x|)\frac{x_1}{|x|}(\phi-\e_m) w\leq
  -\liminf_{m\to\infty}\int_{\partial \Omega_{\e_{m}}} \frac{\partial \phi}{\partial \nu} \frac{\partial w}{\partial x_{1}}\leq0.
 $$
 This is impossible since $V'\not\equiv 0$ by (V3).
\end{proof}

\setcounter{equation}{0}
\section{Proof of Theorem \ref{Thm1.5}}\label{sec5}
Through out this section, we assume $N\geq 2$,  $\alpha>1-N$ and $\sigma\in(0,{2}/{(N-2)^+})$. We first consider the existence and uniqueness of positive radial solution to
\begin{equation}\label{eq5.1}
	-\Delta u + |x|^{\alpha\sigma} u =|u|^{2\sigma}u,\quad u\in H^1(\R^N).
\end{equation}
To this end, we need   a suitable function space. For any $\phi\in C_0^\infty(\R^N)$, define  
$$\|\phi\|_\sigma=\left(\int_{\R^N}|\nabla \phi|^2+|x|^{\alpha\sigma} \phi^2\right)^\half,
$$
and let $X^\sigma_r$ be the completion of  $ \set{\phi\in C_0^\infty(\R^N) |\phi(x)=\phi(|x|)} $ with respect to $\|\cdot\|_\sigma$. 
Note that $X_r^\sigma\subset H_r^1(\R^N)$ when $\alpha\geq 0$ and $X_r^\sigma\supset H_r^1(\R^N)$
when $\alpha<0$.
We have
\begin{Lem}\label{Lem5.1}
	Assume  \eqref{eq1.13}.
	There is $C_N>0$ independent of $\sigma$ such that
	for any $u\in X_r^\sigma$, 
	$$|u(x)|\leq C_N \|u\|_\sigma |x|^{\frac{1-N}2-\frac{\alpha\sigma}4},\quad |x|\neq 0.
	$$
Moreover, for each $p\in[2\sigma+2,2^*)$, $X_r^\sigma$ embeds to $L^{p}(\R^N)$ compactly.
 
\end{Lem}
\begin{proof}
	 By density, for $u \in X_r^\sigma \cap  C_0^\infty(\R^N)$,
$$\|u\|_\sigma^2=\int_{\R^N}|\nabla u|^2+|x|^{\alpha\sigma} u^2\d x
=\omega_N\int_{0}^{+\infty}r^{N-1}\left((u')^2+r^{\alpha\sigma} u^2 \right)\d r,$$
where $\omega_N$ denotes the surface area of the unit sphere in $\mathbb{R}^N$.	 
Since
$$
 \frac{\d}{\d r}\left(r^{N-1+\frac{\alpha\sigma}2} u^{2}\right)=2 u r^{N-1+\frac{\alpha\sigma}2} \frac{\d u}{\d r}  + (N-1+\frac{\alpha\sigma}2)r^{N-2+\frac{\alpha\sigma}2}u^2\geq 
 2 u r^{N-1+\frac{\alpha\sigma}2} \frac{\d u}{\d r}.
$$
Integrating over $[r,+\infty)$, we obtain
$$
\begin{aligned}
	r^{N-1+\frac{\alpha\sigma}2} u^{2}(r)  \leq 2 \int_{r}^{\infty}|s^{\frac{\alpha\sigma}2}u|\left|\frac{\d u}{\d s}\right| s^{N-1} \d s  
  \leq  \int_{r}^{\infty} s^{N-1}\left((u')^2+s^{\alpha\sigma} u^2 \right) \d s\leq C_N^{2}\|u\|_\sigma^2 .  
\end{aligned}
$$
We have verified the first part of Lemma \ref{Lem5.1}.
To show the compactness,  assume $u_n\rightharpoonup 0$ in $X_r^\sigma$. 
Fixing $p\in [2\sigma+2,2^*)$,
we have for  $\kappa\in(0,1)$ and $R>0$ 
\begin{align*}
	\int_{|x|>R} |u_n|^{p}
	=&\int_{|x|>R}\left(|x|^{\frac{\alpha\sigma}2}|u_n|\right)^{2-2\kappa}|x|^{-\alpha\sigma(1-\kappa)}|u_n|^{p-2+2\kappa}\\
	\leq& \left(\int_{\R^N}|x|^{ {\alpha\sigma}}u_n^2\right)^{1-\kappa}
	\left(\int_{|x|>R} |x|^{-\alpha\sigma(\kappa^{-1}-1)}|u_n|^{(p-2)\kappa^{-1}+2} \right)^\kappa
	\\
	\leq& \|u_n\|_{\sigma}^{2(1-\kappa)}\left(\int_{|x|>R} |x|^{-\alpha\sigma(\kappa^{-1}-1)}|u_n|^{(p-2)\kappa^{-1}+2} \right)^\kappa\\
	\leq & C_N^{p-2+2\kappa}\|u_n\|_\sigma^{p}
	\left(\int_{|x|>R} |x|^{(\frac{1-N}2-\frac{\alpha\sigma}4)(p-2)\kappa^{-1}-\alpha\sigma\kappa^{-1}+1-N+\frac{\alpha\sigma}2}  \right)^\kappa.
\end{align*}
Note that by \eqref{eq1.13}, $\frac{1-N}2-\frac{\alpha\sigma}4<0$ and
$$\left(\frac{1-N}2-\frac{\alpha\sigma}4\right)(p-2) -\alpha\sigma\leq \left(1-N-\frac{\alpha\sigma}2-\alpha\right)\sigma<0.$$ 
Thus,  we can fix $\kappa>0$ sufficiently small such that
\[
 \begin{aligned}
	   \left(\frac{1-N}2-\frac{\alpha\sigma}4\right)(p-2)\kappa^{-1}& -\alpha\sigma \kappa^{-1}+1-N+\frac{\alpha\sigma}2\\
	  & \leq\left(1-N-\frac{\alpha\sigma}2-\alpha\right)\sigma\kappa^{-1}+1-N+\frac{\alpha\sigma}2
	 <-N. 
 \end{aligned}
 \]
Then we have 
$$\lim_{R\to\infty}\sup_{n}\int_{|x|>R} |u_n|^{p}=0.
$$
We can finally verify $u_n\to 0$ in $L^{p}(\R^N)$ by the compact embedding from $X_r^\sigma$ to $L_{loc}^{p}(\R^N)$.
\end{proof}
The existence is a corollary of Lemma \ref{Lem5.1} and the uniqueness is a result 
of a modified decay estimate and the arguments in \cite{BO}. 
\begin{Pro}\label{Pro5.1} Assume \eqref{eq1.13}.
There is a positive solution $u_\sigma\in X_\sigma^r$ to \eqref{eq5.1} satisfying
\begin{equation}\label{eq523}
	 u_\sigma(x)\leq C_\sigma \exp\big( -c_\sigma|x|^{\frac{\alpha\sigma+2}2} \big),
\end{equation}
 where $C_\sigma, c_\sigma$ are positive constants independent of $x$.  Assuming further that $-\sigma\alpha<1$, then $u_\sigma$ is the 
 unique positive radial solution to \eqref{eq5.1}.
\end{Pro}
\begin{proof}
(i) By the compact embedding $X_r^\sigma\subset L^{2\sigma+2}(\R^N)$
 and the mountain pass theorem,    the following equation has a nonnegative nontrivial radial solution
 $u\in X_r^\sigma$,
\begin{equation} \label{eq5.24}
	-\Delta u + |x|^{\alpha\sigma} u =|u|^{2\sigma}u,\quad \limsup_{|x|\to\infty}|x|^{\frac{N-1}2+\frac{\alpha\sigma}4} u(|x|)<+\infty.
\end{equation}
By \eqref{eq1.13}, $u\in C(\R^N)$ since there holds $|x|^{\alpha\sigma}\in L^q(B_1(0))$ for some $q>N/2$.
By the maximum prinviple, we can conclude that $u$ is positive in $\R^N\setminus\{0\}$.
In order to show that $u$ is a solution of \eqref{eq5.1}, we claim that  \eqref{eq523},
which  ensures   the  square integrablity, holds for any solution to \eqref{eq5.24}.
By \eqref{eq1.13}, we know $\frac{1-N}2-\frac{\alpha\sigma}4<\frac\alpha2$.
Then by Lemma \ref{Lem5.1}, there is $R>0$ such that if $|x|\geq R$, then
$$-\Delta u +\frac{1}2|x|^{\alpha\sigma}u\leq -\Delta u + (|x|^{\alpha\sigma}-|u|^{2\sigma})u=0.
$$
On the other hand, fix $c_\sigma\in(0,\frac{\sqrt2}{2+\alpha\sigma})$ and denote
$v(x)=\exp\big( -c_\sigma|x|^{\frac{\alpha\sigma+2}2} \big)$.
Then there is $R'\geq R$ such that if $|x|\geq R'$,then
$$-\Delta v +\frac{1}2|x|^{\alpha\sigma}v\geq 0.
$$
Let $C_\sigma>0$ be such that $u(x)\leq C_\sigma v(x)$ for $|x|=R'$.
We have  
\begin{align*}
	0\leq&\int_{|x|>R'}|\nabla (u-C_\sigma v)^+|^2+\frac12 |x|^{\alpha\sigma}((u-C_\sigma v)^+)^2\\
	=&\int_{|x|>R'}\left(-\Delta(u-C_\sigma v)+\frac12|x|^{\alpha\sigma}(u-C_\sigma v)\right)(u-C_\sigma v)^+\leq 0.
\end{align*}
This implies $u\leq C_\sigma v$.

(ii)
If assume further that $-\sigma\alpha<1$, then by the regularity theory,  $u_\sigma\in C^{2-\gamma'}(\R^N)\cap C^2(\R^N\setminus\{0\})$ for 
each $\gamma'\in(-\alpha\sigma,1)$. Then $u_\sigma'(0)=0$ and similar to the proof of Proposition \ref{LemA1} (i),  we can conclude that $u_\sigma>0$.
To  see the uniqueness, we   check the arguments of \cite[Theorem 1.1]{BO} to 
\begin{equation}\label{eq525}
	 \begin{cases}
   	u''+\frac{N-1}r u'- r^{\alpha\sigma}u+ |u|^{2\sigma}u=0,\quad u>0\quad  \mbox{in} \quad (0,+\infty),\\
   	u'(0)=0,\quad \limsup_{r\to\infty}r^{\frac{N-1}2+\frac{\alpha\sigma}4} u(r)<+\infty.
   \end{cases}
\end{equation}
We first check that under condition \eqref{eq1.13},  $V(r)=r^{\alpha\sigma}$ satisfies the second assumption of \cite[(V2)]{BO}, 
i.e., the function
$$H(r)=\left(\frac{2(N-1)}{\sigma+2}+\alpha\right)\sigma r^{\alpha\sigma+2}-D(N,\sigma),$$
where $D(N,\sigma)$ is a negative constant if $N=2$ and a positive constant if $N\geq 3$,
satisfies that
$\inf _{r>0} H(r)>0$ for $N=2$, and $H$ has the unique simple zero in $(0, \infty)$ and $\limsup _{r \rightarrow 0} H(r)<0$ for $N \geqslant 3$.
In fact, the condition \eqref{eq1.13} implies that $H(r)$ is strictly increasing, $\lim_{r\to\infty}H(r)=+\infty$ and $H(0)>0$ if $N=2$ and $H(0)<0$ if $N\geq 3$.

Now the only   point  to apply \cite[Theorem 1.1]{BO} for \eqref{eq525} is that $V(r)=r^{\alpha\sigma}$ does not meet the assumption \cite[(V1)]{BO}. However checking the proof \cite[Theorem 1.1]{BO}, we can see that the fact $u(r)\leq C\exp\big( -cr^{\frac{\alpha\sigma+2}2} \big)$ 
is sufficient to continue the arguments.
\end{proof} 
Now we consider the asymptotic behavior of the unique positive radial solution $u_\sigma$
as $\sigma\to 0^+$.
Set $$v_\sigma=\sigma^{\frac{\alpha}{4+2\alpha\sigma}}u_\sigma( {\sigma}^{-\frac{1}{2+\alpha\sigma}}\cdot).$$  Then $v_\sigma$ satisfies
\begin{equation}\label{eq5.2}
	-\Delta v + \sigma^{-1}|x|^{\alpha\sigma} v=  \sigma^{-1}|v|^{2\sigma} v,
\end{equation}
or  equivalently, 
\begin{equation}\label{eq5.3}
	-\Delta v + \sigma^{-1}( |x|^{\alpha\sigma}-1) v
	=\sigma^{-1} (|v|^{2\sigma}- 1) v.
\end{equation}
Consider
 the functional corresponding to \eqref{eq5.2}
$$I_{1,\sigma}(v)=\frac12\int_{\R^N}(|\nabla v|^2+ \sigma^{-1}|x|^{\alpha\sigma} v^2)-\frac1{\sigma(2\sigma+2)}\int_{\R^N}|v|^{2\sigma+2},
\quad v\in X_r^\sigma.
$$
By the standard argument, we  can see  
$$I_{1,\sigma}(v_\sigma)=\inf_{v\in X_r^\sigma\setminus\{0\}}\sup_{t\geq 0} I_{1,\sigma}(tv).
$$

We define 
$$I_1(v)=\frac12\int_{\R^N}\left(|\nabla v|^2+\alpha v^2   \log |x|   \right)\d x
+\frac1{2}\int_{\R^N} v^{2}\d x\\
-\frac1{2}\int_{\R^N}  v^2 \log v^2 \d x.
$$
where $v$ belongs to
$$Y_r:=\Set{v\in H_r^1(\R^N) | \int_{|x|\geq 1}v^2\log |x| <+\infty},
$$
equipped with the norm $$\|v\|_Y:=\left(\int_{\R^N}|\nabla v|^2+(1+(\log|x|)^+)v^2 \right)^\half.$$
Note that 
\begin{equation}\label{eq5.7}
	\lim_{\sigma\to0^+}I_{1,\sigma}(v)=I_1(v)\quad \mbox{for each}\quad v\in C_0^1(\R^N)\cap H_r^1(\R^N). 
\end{equation}
\begin{Pro}
Let $\alpha>1-N$. Then 
\begin{equation}\label{eq5.8}
		0<\liminf_{\sigma\to0^+} I_{1,\sigma}(v_\sigma)\leq \limsup_{\sigma\to0^+} I_{1,\sigma}(v_\sigma)<+\infty,
\end{equation}
\begin{equation}\label{eq5.9}
	0<\limsup_{\sigma\to0^+} \|v_\sigma\|_Y\leq \limsup_{\sigma\to0^+} \|v_\sigma\|_Y<+\infty,
\end{equation}
\end{Pro}
\begin{proof}
Fix a nonnegative nontrivial function $v\in  C_0^1(\R^N)\cap H_r^1(\R^N)$.
We can find $T_0>0$ such that 
$I_1(T_0v)< 0$. 
By \eqref{eq5.7}, for $\sigma>0$ sufficiently small $I_{1,\sigma}(T_0v)<0$. 
Note that the function
$t\mapsto I_1(tv)$ admits a unique extreme point in $(0,+\infty)$. 
Then we have 
$$\lim_{\sigma\to0^+}I_{1,\sigma}(v_\sigma)\leq \lim_{\sigma\to0^+} \max_{t\geq 0}I_{1,\sigma}(tv)=\lim_{\sigma\to0^+} \max_{t\in[0,T_0]}I_{1,\sigma}(tv)=\  \max_{t\geq 0}I_{1}(tv)<+\infty.
$$
Let us  rewrite
\begin{equation}\label{eq510}
	\begin{aligned}
	I_{1,\sigma}(v)&=\frac12\int_{\R^N}\left(|\nabla v|^2+ \sigma^{-1}(|x|^{\alpha\sigma}-1)v^2\right)
	 +
	 \frac1{2\sigma+2}\int_{\R^N} |v|^{2\sigma+2}
	 -\frac1{2}\int_{\R^N}\sigma^{-1}(|v|^{2\sigma}-1)v^2\\
	\geq&\frac12I_{1,\sigma}'(v)v=
	\frac12\int_{\R^N}\left(|\nabla v|^2+ \sigma^{-1}(|x|^{\alpha\sigma}-1)v^2 \right)
	 -\frac1{2}\int_{\R^N}  \sigma^{-1}(|v|^{2\sigma}-1)v^2 \\
	&\qquad\qquad\;
	\geq \frac12\int_{\R^N}\left(|\nabla v|^2+  \alpha v^2  \log |x|   \right)
	 -\frac1{2}\int_{\R^N}  \sigma^{-1}(|v|^{2\sigma}-1)v^2 :=I_{2,\sigma}(v),
	\end{aligned}
\end{equation}
in which the last inequality holds because 
\begin{equation}\label{eq511}
	\sigma^{-1}(s^{t\sigma}-1) \geq t \log s\quad  \mbox{for each}\quad (t,s,\sigma) \in\R\times (0,+\infty) \times(0,+\infty). 
\end{equation}
The functional $I_{2,\sigma}(v)$ is well defined on
$Y_r$.
By the same arguments to Lemma \ref{lemma2.1}, there is   $\mu_1>0$ such that 
\begin{equation}\label{eq5.10}
	 \int_{\R^N}\left(|\nabla v|^2+  (\alpha-\alpha_0) v^2  \log |x| +\mu_1 v^2  \right)\d x\geq \half \|v\|_Y^2,\quad v\in Y_r.
\end{equation}
Here we recall that $\alpha_0$ given by \eqref{alpha1} satisfies $\alpha_0\leq 0$ and $1-N<\alpha_0<\alpha$.
On the other hand, similar to \eqref{eq30}, for fixed $p\in (2,2^*)$  satisfying \eqref{eq2.16}  and $p_1=p-\frac{(p-2)\alpha_0}{1-N}\in(2,p]$, when $\sigma$ is small enough,
we have
\begin{equation}\label{eq5.11}
	\begin{aligned}
		\sigma^{-1}&(|v|^{2\sigma}-1)v^2-\alpha_0v^2\log |x|+\mu_1v^2\\
		&= \int_0^1(|v|^{2s\sigma+2}-v^2)\log v^2\d s
	+ v^2\log(v^2|x|^{-\alpha_0}e^{\mu_1})
		\\
		&\leq \max\{0,(|v|^{2\sigma+2}-v^2)\log v^2\}
	+ v^2\log(v^2|x|^{-\alpha_0}e^{\mu_1})
		\\
		 &\leq
	C_3' \left(|v|^p+\|v\|_{H^1(\R^N)}^{p-p_1} |v|^{p_1}\right)\quad \mbox{in}\quad\R^N,
	\end{aligned}
\end{equation}
where $C_3'$ is a constant independent of $\sigma$.
Then by \eqref{eq5.10} and \eqref{eq5.11}, for some $C>0$ independent of $\sigma$,
there holds
$$I_{2,\sigma}(v)\geq \frac14\|v\|_Y^2-C\|v\|_Y^p.
$$
Hence there is $m>0$ independent of $\sigma$ such that
 \begin{equation}\label{Isig}
	 I_{2,\sigma}(v)\geq \frac18 m^2\quad  \mbox{for}\quad \|v\|_Y=m,\quad\mbox{and}\quad
	   I_{1,\sigma}'(v)v\geq \frac14\|v\|_Y^2 \quad \mbox{for}\quad \|v\|_Y\leq m.
 \end{equation}
By Proposition \ref{Pro5.1}, $v_\sigma\in Y_r$.
Then for small $\sigma>0$
$$I_{1,\sigma}(v_\sigma) =\max_{t\geq 0}I_{1,\sigma}(tv_\sigma)\geq \max_{t\geq 0}I_{2,\sigma}(tv_\sigma)\geq I_{2,\sigma}(m\|v_\sigma\|_Y^{-1} v_\sigma)
\geq \frac18 m^2.
$$
This completes the proof of \eqref{eq5.8}. Note that by \eqref{Isig}, we have $\|v_\sigma\|_Y> m$. Otherwise, $0=I_{1,\sigma}'(v_\sigma)v_\sigma\geq \frac14\|v_\sigma\|_Y^2.$ This is a contradiction. Then,
 \[
	\liminf_{\sigma\to0^+} \|v_\sigma\|_Y\geq m>0.
 \]
To show the boundedness of $ \|v_\sigma\|_{Y} $, we note that
$$\limsup_{\sigma\to0^+}\frac1{2\sigma+2}\int_{\R^N} v_\sigma^{2\sigma+2}=\limsup_{\sigma\to0^+}\left(I_{1,\sigma}(v_\sigma)-\frac12I_{1,\sigma}'(v_\sigma)v_\sigma\right)<+\infty.
$$
Note that by the Gagliardo - Nirenberg interpolation inequality
$$\|v_\sigma\|_{L^p(\R^N)}\leq C\|v_\sigma\|_{L^{2\sigma+2}}^{1-\nu_{p,\sigma}} \|\nabla v_\sigma\|_{L^2(\R^N)}^{\nu_{p,\sigma}}\leq C\|\nabla v_\sigma\|_{L^2(\R^N)}^{\nu_{p,\sigma}},
$$
where 
$$\nu_{p,\sigma}=\frac{N(p-2\sigma-2)}{p(2\sigma+2-N\sigma)}\to \nu_p\quad\mbox{as}\quad \sigma\to0^+,$$
$\nu_p$ is as in \eqref{eq:2.14} and $C>0$ is a constant independent of $\sigma$ small.
Then by \eqref{eq510}, \eqref{eq5.10} and \eqref{eq5.11}, we have
\begin{equation*}
	 I_{1,\sigma}(v_\sigma)
	\geq \frac{1}{4}\|v_\sigma\|_Y^2
	-C  (\| v_\sigma\|_{Y}^{p\nu_{p,\sigma}}+\| v_\sigma\|_{Y}^{p-p_1+p_1\nu_{p_1,\sigma}}),
\end{equation*}
where $\nu_{p_1,\sigma}\to \nu_{p_1}$ as $\sigma\to0^+$.
 Since $p-p_1+\nu_{p_1} p_1 < p-2+ \nu_pp<2$ by \eqref{eq2.16}, we have
\eqref{eq5.9}.
\end{proof}
Now we are ready to complete the proof of Theorem \ref{Thm1.5}.
\begin{proof}[Proof of Theorem \ref{Thm1.5}]
	(i) and (ii) follow from Proposition \ref{Pro5.1}. We show (iii).
	By \eqref{eq5.9}, up to a subsequence, we have $v_\sigma\rightharpoonup v$ in $Y_r$. Clearly $v$ is a solution to \eqref{eq1.11'}.
	We show $v_\sigma\to v$ in $Y_r$.
	To achieve this, we claim that there is $C>0$ independent of $\sigma$ small
	such that 
	\begin{equation}\label{eq5.18}
		v_\sigma(x)\leq Ce^{-|x| }.
	\end{equation}
	In fact, by \eqref{eq5.3} and \eqref{eq511}, we have \[-\Delta v_\sigma +(\alpha-\alpha_0)v_\sigma\log |x| \leq v_\sigma\left(\sigma^{-1}(v_\sigma^{2\sigma}-1)-\alpha_0\log |x|\right).\]
	 Similar to \eqref{eq5.11}, we have by the boundedness of $\|v_\sigma\|_{H^1(\R^N)}$,
	 \[-\Delta v_\sigma +(\alpha-\alpha_0) v_\sigma\log |x|\leq C_4(v_\sigma^p+v_\sigma^{p_1}) .\]
Since $v_\sigma(x)\to 0$ uniformly by \eqref{radial}, we can get \eqref{eq5.18} by the comparison
arguments.
Now  from \eqref{eq5.18} and the Sobolev inequality, we have in fact 
$$\int_{\R^N} \sigma^{-1}(v_\sigma^{2\sigma}-1)v_\sigma^2\to \int_{\R^N} v^2\log v^2
\quad \mbox{and}\quad \int_{\R^N}\alpha v_\sigma^2\log|x|  \to \int_{\R^N}\alpha v^2\log |x|.
$$
Then we have 
\begin{align*}
	\limsup_{\sigma\to0^+}\int_{\R^N}|\nabla v_\sigma|^2=&\limsup_{\sigma\to0^+}\left(\int_{\R^N} \sigma^{-1}(v_\sigma^{2\sigma}-1)v_\sigma^2
	-\int_{\R^N}\sigma^{-1}v_\sigma^2(|x|^{\alpha\sigma}-1)\right)\\
	\leq& \limsup_{\sigma\to0^+}\left(\int_{\R^N} \sigma^{-1}(v_\sigma^{2\sigma}-1)v_\sigma^2
	-\int_{\R^N}\alpha v_\sigma^2\log |x|\right)\\
	=&\int_{\R^N} v^2\log v^2-\int_{\R^N}\alpha v^2\log |x|=\int_{\R^N}|\nabla v|^2.
\end{align*}
Then $v_\sigma \to v$ in $Y_r$ up to a subsequence. By \eqref{eq5.9}, $v$ is nontrivial and hence positive.
By Theorem \ref{Thm1.2}, the convergence is  independent of subsequences.
Hence $v_\sigma\to v$ in $Y_r$ as $\sigma\to 0^+$. By the regularity theory,
for any $\gamma''\in (0,1)$,
$v_\sigma$ is bounded in  $C^{2- {\gamma''}/2}(\R^N)\setminus C^{2+\gamma''}(\R^N\setminus B_{\gamma''}(0))$.
Hence by Arz\'ela-Ascoli theorem,
$v_\sigma\to v$ in $C^{2-\gamma''}(\R^N)\setminus C^2(\R^N\setminus B_{\gamma''}(0))$.
By uniqueness of $v$ and $v_\sigma$, the convergences are independent of subsequences.
At last, we remark that
$$ \sigma^{\frac{\alpha}{4 }}u_\sigma( {\sigma}^{-\frac{1}{2 }}\cdot)-v_\sigma=
\sigma^{\frac{\alpha^2\sigma}{8+4\alpha\sigma}}v_\sigma(\sigma^{-\frac{\alpha\sigma}{4+2\alpha\sigma}}\cdot)-v_\sigma\to 0\quad \mbox{in  $Y_r$ and $C^{2-\gamma''}(\R^N)\setminus C^2(\R^N\setminus B_{\gamma''}(0))$}
$$
to complete the proof.
\end{proof}

\appendix
\setcounter{equation}{0}
\section{Appendix}\label{secA}
In this section, assuming (Vab), (V1) and  $\theta \in\Theta$, we give several technical results
which are useful to show uniqueness
on the following equation
\begin{equation}\label{A4.1}
	\left\{
		\begin{aligned}
			 &u''+\frac{N-1}{r}u'- V_\delta(r) u  + B_\delta(r) u\log u^2=0,\quad u>0\quad\text{in}\quad (0,+\infty)\\
			 &u'(0)=0,   r^{-\frac\theta{2} }u(r)\to 0\ \text{as}\ r\to+\infty.
		\end{aligned}
		\right.
	\end{equation}
 
	\begin{Lem}\label{Cor2.1}
		Let $u(r)$  solves \eqref{P3}. Then there is $R>0$ such that $u'(r)<0$ {for} 
		$r\geq R$.	 Moreover,
		$$\lim_{r\to\infty} r^{N-1}u'(r)=\liminf_{r\to\infty} r^{N-1} V_\delta(r)u(r)=0. 
		$$
	\end{Lem}
		\begin{proof}
		  By Proposition \ref{LemA1},
		  we can assume that $R$ is sufficiently large such that for $r\geq R$,
		  \begin{equation}\label{eqA11}
			  (r^{N-1} u')'=r^{N-1}W(r)>0,
		  \end{equation}
	  where $W(r):=V_\delta(r) u(r)  - B_\delta(r) u(r)\log u^2(r).$
	  Then $r^{N-1} u'$ is strictly increasing in $[R, +\infty)$.
	  If $u'\left(r_{1}\right) \geq 0$ for some $r_{1}>R$,
	  then for $r\geq r_1$, we have $u'(r)> r^{1-N}r_{1}^{N-1} u'(r_{1})\geq 0$.
	  Hence for $r\geq r_1$ we obtain a contradiction that
	   $0<u(r_1)< u(r)\to0$ as $r\to+\infty$. 
	  Thus, we see that $u'(r)<0$ for $r>R$. 
	  
		To show $\lim _{r \rightarrow \infty} r^{N-1} u'(r) =0$. 
		By contradiction and the monotonicity of $r^{N-1}u'$,
		we assume that there are $\delta>0$ and $r_2>R$ such that 
		$r^{N-1} u'(r)<-\delta$ for $r\geq r_2$.
	  Then as $r\to \infty$, we get a contradiction that
	  \[1\geq e^ru(r)=-e^r\int_{r}^{+\infty} u'(s)\d s> \delta e^r\int_{r}^{2r} s^{1-N} \d s\to+\infty.\]
	  
	  Since $r^{N-1} u'(r) \to 0$  as $r\to \infty$, by \eqref{eqA11}, we can find
	   $r_n\in (R+n-1, R+n)$, such that 
	   $$r_n^{N-1}W(r_n)=\int_{R+n-1}^{R+n}r^{N-1}W(r)\d r\to 0.
	   $$
	   
	Together with \eqref{eqA11} and Proposition \ref{LemA1},  we have 
\[	 0=\liminf_{r\to\infty} r^{N-1}W(r)=\liminf_{r\to\infty} r^{N-1}V(r)u(r). \qedhere
\]
 \end{proof}

\begin{Lem}\label{lem4.2}
	
	Suppose that $u_1,u_2$ are two solutions of \eqref{A4.1} satisfying $u_1(0)<u_2(0)$,
	then  
	$$\rho:=\inf\set{r\in(0,\infty) | u_1(r)> u_2(r)}
	 $$ exists and
	satisfies
	$u_1(\rho)=u_2(\rho)$,  and
	$$u_1<u_2,\quad 
	\frac{\d}{\d r}\left(\frac{u_1}{u_2}\right)>0\quad \mbox{in} \ (0,\rho).
	$$
	Moreover if
	$$u_1>u_2>0\ \text{in}\ (\rho,\infty),
	$$
	then 
	$$\frac{\d}{\d r}\left(\frac{u_1}{u_2}\right)>0\ \text{in} \ (0,\infty).
	$$
	\end{Lem}
	\begin{proof}
	The proof is similar with \cite[Lemma 1.2]{KT}. 
	To show the existence of $\rho\in(0,\infty)$, we only need to prove
	\begin{equation}\label{sig}
		\set{r\in(0,\infty) | u_1(r)> u_2(r)}\neq\emptyset.
	\end{equation}
	We note that
	\begin{equation}\label{ddr}
	\begin{aligned}
	\frac{\d }{\d r}\left((r^{N-1}u_2^2)\frac{\d }{\d r}\bigg(\frac{u_1}{u_2}\bigg)\right)&=\frac{\d}{\d r}\left(r^{N-1}(u_1'u_2-u_2'u_1)\right)\\
	&=(r^{N-1}u_1')'u_2-(r^{N-1}u_2')'u_1\\
	&=B_\delta r^{N-1} u_1u_2[(\log u_2^2-\log u_1^2)].
	\end{aligned}
	\end{equation}
	We know from $u_1'(0)=u_2'(0)=0$, Proposition \ref{LemA1} (ii)  and Lemma \ref{Cor2.1} that 
	$$\lim_{r\to0}(r^{N-1}u_2^2)\frac{\d}{\d r}\left(\frac{u_1}{u_2}\right)=0,\quad
	 \lim_{r\to\infty} (r^{N-1}u_2^2)\frac{\d}{\d r}\left(\frac{u_1}{u_2}\right) =0.
	$$
	If \eqref{sig} is not true, then  by \eqref{ddr}
	\begin{equation*}\frac{\d }{\d r}\left((r^{N-1}u_2^2)\frac{\d }{\d r}\bigg(\frac{u_1}{u_2}\bigg)\right)\geq0.
	\end{equation*}
	So,
	$$ (r^{N-1}u_2^2)\frac{\d }{\d r}\bigg(\frac{u_1}{u_2}\bigg)\equiv 0\quad \mbox{in}\ (0,+\infty).
	$$
	This contradicts to
	$$\frac{\d }{\d r}\big[(r^{N-1}u_2^2)\frac{\d }{\d r}(\frac{u_1}{u_2})\big]=B_\delta r^{N-1} u_1u_2[(\log u_2^2-\log u_1^2)]\not\equiv 0.
	$$   
	If $u_1>u_2$ in $(\rho,+\infty)$, then by \eqref{ddr}, the function
	$(r^{N-1}u_2^2)\frac{\d}{\d r}(\frac{u_1}{u_2})$ is strictly increasing in $(0,\rho)$ and strictly decreasing in $(\rho,+\infty)$.
	Therefore, $\frac{\d}{\d r}(\frac{u_1}{u_2})>0$ in $(0,+\infty)$.
	\end{proof}
	From the proof of Lemma \ref{lem4.2} we can also get:
	\begin{Lem}\label{lemA2'}
	Suppose $\rho>0$ and  $u_1,u_2$ are two solutions of 
	\begin{align*}
			 &u''+\frac{N-1}{r}u'-V_\delta u+B_\delta u\log u^2=0,\quad \text{in}\ (0,\rho),\\
			 &u'(0)=0,\quad u>0 \quad \text{in}\ (0,\rho)
		\end{align*}
	such that 
	$$u_1<u_2\quad \text{in}\ [0,\rho).
	$$
	Then
	$$\frac{\d}{\d r}\left(\frac{u_1}{u_2}\right)>0\quad \text{in} \ (0,\rho).
	$$
	\end{Lem}
	\begin{Lem}\label{lemA3}
	For $b_0>0$, let $w$ be the  unique solution to
		the initial value problem
	\begin{equation}
		\label{eq4.9}w''+\frac{N-1}{r}w'+b_0w=0,\quad  w(0)=1,\quad  w'(0)=0.
	\end{equation}
	 Then there is $r_1>0$  such that
	$w(r_1)=0$,    
	$w>0$ in $(0,r_1)$ and $w'<0$ in $(0,r_1]$.
	\end{Lem}
	\begin{proof}
	Let $\lambda_1$ be the first eigenvalue of
	\begin{equation*}
		 \left\{
			\begin{aligned}
				-\Delta \phi&=\lambda_1\phi	&&\mbox{in}\quad B_1(0),\\
				\phi&=0 &&\mbox{on}\quad \partial B_1(0),
			\end{aligned}
			\right.
	\end{equation*}	 
	with the corresponding eigenfunction $\phi(x)=\phi(\abs x)$ which satisfies
	 $\phi(r)>0$ in $(0,1)$, $\phi(1)=0$, $\phi'(0)=0$ and $\phi'(r)<0$, $0<r\leq 1$. 
	 Then
	$\psi=\phi^{-1}(0)\phi(r_1^{-1}\cdot)$ with $r_1:= b_0^{-\half}{\lambda_1}^\half$ satisfies \eqref{eq4.9}
	and coincides with  $w$ in $(0,r_1]$.
	\end{proof}
	For $\beta>0$, we consider the following initial value problem:
\begin{equation}\label{IV}
    \left\{
    \begin{aligned}
         &u''+\frac{N-1}{r}u'-V_\delta(r)u+B_\delta(r)u\log u^2=0,\\
         &u(0)=\beta,\quad
         u'(0)=0.
    \end{aligned}\right. 
    \end{equation}
	By Remark \ref{Rek2.1}, we can denote the unique solution by  $u(r;\beta)$.
    The next lemma sketches  the  graph of $u(r;\beta)$ for large $\beta$.
	\begin{Lem}\label{lemA4}
		Denoting
		$$v(r;\beta):= \beta^{-1}u((\log \beta)^{-\half}r;\beta),
		$$
		we have $v(r;\beta)\to w(r)$ in $C_{loc}([0,+\infty))\cap C^1_{loc}((0,\infty))$ as $\beta\to\infty$, where $w$ is the solution to
		the problem \eqref{eq4.9} with $b_0=2B_\delta(0)$ in Lemma \ref{lemA3}.
	\end{Lem}
	\begin{proof}
		Noting that $v(r;\beta)$ satisfies
		$v(0)=1, v'(0)=0$ and
		\begin{align*}
		 (r^{N-1}v')'+2 B_\delta((\log \beta)^{-\frac12}r)r^{N-1} v 
		 =(\log \beta)^{-1}r^{N-1} \left(V_\delta((\log \beta)^{-\half} r) v -  B_\delta((\log \beta)^{-\half} r) v\log v^2\right).
		\end{align*}
		We claim that  $v(\cdot;\beta)$ is bounded in $C_{loc}([0,+\infty))$ for $\beta\geq e$, i.e., for any $T>0$,
		$\max_{r\in[0,T]} |v(r;\beta)|$ is bounded for $\beta\geq e$.
		It suffices to show that if 
		$\max_{r\in[0,r_0]} |v(r;\beta)|\leq C$ for some $r_0\geq 0$ and $C>0$, then   
		there is $\delta>0$ independent of $\beta$
		such that 
		\begin{equation}\label{rdelta}
			\max_{r\in[0,r_0+\delta]} |v(r;\beta)|\leq C+1.
		\end{equation}
		Without loss of generality, we assume that 
		\[r_1:=\inf\set{r>r_0| \abs{v(r;\beta)-v(r_0;\beta)}>1}<+\infty.
		\]
		We estimate the lower bound for $r_1$ to show \eqref{rdelta}. 
		Notice that $v(r;\beta)$ satisfies
		\begin{equation}\label{v'}
			\begin{aligned}
				v'(r)=r^{1-N}\int_0^r s^{N-1}\Big(&(\log \beta)^{-1}V_\delta((\log \beta)^{-\half}s)v
				\\&-(\log \beta)^{-1}B_\delta((\log \beta)^{-\half} s)v\log v^2 
				- 2B_\delta((\log \beta)^{-\half} s)v\Big)\d s.
			\end{aligned}
		\end{equation}
		We have by \eqref{v'} and (V1) that
		\[
			\begin{aligned}
				1=&|v(r_1)-v(r_0)|=\left|\int_{r_0}^{r_1} v'(r) \d r\right|\\
				\leq &C\int_{r_0}^{r_1}r^{1-N} \d r \int_0^r s^{N-1} (\log \beta)^{-1} |V_\delta((\log \beta)^{-\half} s)|\d s 
				+ 3(1+|C\log C|) \int_{r_0}^{r_1}r \d r  \\
				\leq &C\int_{r_0}^{r_1}(\log \beta)^{\frac{N-2}2} r^{1-N} \d r \int_0^{r(\log\beta)^{-\half} }s^{N-1} |V_\delta(s)|\d s 
				+ 3(1+|C\log C|) \int_{r_0}^{r_1}r \d r  \\
				\leq &C\int_{r_0}^{r_1}(\log \beta)^{-\frac12}  \|V_\delta(|\cdot|)\|_{L^N(|x|\leq r_1(\log\beta)^{-\half})} \d r
				+ C'(r_1-r_0)^2  \\
				\leq & C'\left(\|V_\delta(|\cdot|)\|_{L^N(|x|\leq r_1(\log\beta)^{-\half})}(r_1-r_0)+(r_1-r_0)^2\right),
			\end{aligned}
		\]
		where $C'$ is a constant independent of $\beta$.  Then we have proved \eqref{rdelta} and
		 verified the claim.
		 By \eqref{v'} again, we can check that for any $T>0$,
		 \[\max_{r\in [0,T]}|v'(r)|\leq  C \|V_\delta(|\cdot|)\|_{L^N(|x|\leq T(\log\beta)^{-\half})} 
		 + 3(1+|C\log C|) T.
			 \]
		So $v(r;\beta)$ is bounded in $C_{loc}^1([0, +\infty))$, and thus by the equation, bounded
		in $C^2_{loc}((0,+\infty))$.
		Hence, by the Arz\'ela-Ascoli theorem, as $\beta\to\infty$,
		$v(r;\beta)$ converges in $C_{loc}([0,+\infty))\cap C^1_{loc}((0,\infty))$ to the unique solution of
		\eqref{eq4.9} with $b_0=2B_\delta(0)$.
		\end{proof}
		\begin{Lem}\label{prop4.1}
			Suppose that problem \eqref{A4.1} has two distinct 
			solutions $u_1,u_2$ such that $u_1(0)<u_2(0)$. Then there exists
			a solution $u_3$ of \eqref{A4.1}  such that
			\begin{equation}
			u_3(0)\geq u_2(0) \quad \text{and}\quad
			\#\set{r\in(0,\infty) | u_1(r)=u_3(r)}= 1.\label{A6}
			\end{equation}
		\end{Lem}
		\begin{proof}
			The proof is similar to \cite[Appendix A]{KT} and \cite[Lemma A.3]{BO}.
			Let $u(r;\beta)$, $\beta > 0$ be the solution of the initial value problem
			\eqref{IV}. 
			For $\beta\geq u_2(0)$, let
			$$n(\beta)=\#\set{r>0 | u(r;\beta)=u_1(r)}.
			$$
			By Lemma \ref{lem4.2}, we may assume
			$n(u_2(0))\geq 2$.
			Then by the continuous dependence of solution to initial value,
			$n(\beta)\geq 2$ for $\beta>u_2(0)$ sufficiently close to $u_2(0)$.
			Denote 
			$$\beta^*:=\sup\set{\beta>u_2(0) | n(\widetilde\beta)\geq 2\
			 \mbox{for all}\ \widetilde{\beta}\in (u_2(0), \beta)}.
			$$
			For $\beta\in(u_2(0), \beta^*)$ let $\rho_1(\beta)$ and $\rho_2(\beta)$ be the first and 
			the second intersection points of $u(\cdot;\beta)$ and $u_1$.
			Then we have $u(\cdot;\beta)>0$ in $(0, \rho_2(\beta))$ for $\beta\in(u_2(0),\beta^*)$.
			Otherwise, we set $$\beta_0:=\inf\set{\beta\in(u_2(0),\beta^*)| u(r;\beta)\leq 0\ \mbox{for some}\ r}.$$ 
			Then the solution curve of $u(r;\beta_0)$ is tangent to the $r$-axis, which is impossible 
			by the uniqueness of the initial value problem.
			\noindent
			By Lemma \ref{lemA4}, we can find  $\bar\beta>  u_2(0)$ , such that
			\begin{align*}
				\mbox{(i)}&\quad   u(r;\bar\beta)\  \text{hits zero at some}\ r_0\in (0,\infty),\\
				\mbox{(ii)}&\quad   \#\set{r\in(0,r_0) | u(r;\bar\beta)=u_1(r)}=1.
			\end{align*}
			The existence of $\bar\beta$ implies $\beta^*<\infty$.
			Next we will prove that $u_3(r):= u(r;\beta^*)$ is the solution of \eqref{A4.1}
			 satisfying  \eqref{A6}.
			
			 \textbf{Case 1.} $\rho_1(\beta^*)$ does not exist, i.e. $\rho_1(\beta)\to\infty$ as $\beta\to\beta^*$.This case will not happen. 
			 In fact,
			 for $\beta\in(u_2(0),\beta^*)$, we have, by Lemma \ref{lemA2'},
			$$ \frac{\d}{\d r}\left(\frac{u_1}{u(\cdot;\beta)}\right)>0\quad \text{in} \ (0,\rho_1(\beta)).
			$$
			Then
			 $$ u_1<u(\cdot;\beta)<\frac{\beta}{u_1(0)}u_1 \quad\text{in } [0,\rho_1(\beta)).
			$$
			Taking limit as $\beta\to\beta^*$, we get
			$$u_1\leq u_3\leq \frac{\beta^*}{u_1(0)}u_1\quad \text{in }(0,\infty),
			$$
			which means $u_3$ solves \eqref{A4.1}. 
			By the uniqueness of initial value problem,
			 $u_3>u_1$.  This contradicts to Lemma \ref{lem4.2}.
			
			\textbf{Case 2.} $\rho_1(\beta^*)$ exists but $\rho_2(\beta^*)$ does not exist,
			i.e. $\rho_2(\beta)\to \infty$ as $\beta\to \beta^*$.
			
			In this case, for $\beta\in(u_2(0),\beta^*)$ we have
			\begin{align*}
			&0<u_1(r)<u(r;\beta)\quad \text{in }(0,\rho_1(\beta)),\\
			&0<u(r;\beta)<u_1(r)\quad \text{in }(\rho_1(\beta),\rho_2(\beta)).
			\end{align*}
			Taking   limits as $\beta\to\beta^*$, we get our desired $u_3$.
			\end{proof}
			\vskip .1in
			\noindent{\bf Acknowledgement.} 
			This work is supported by NSFC-12001044, NSFC-12071036 and the Fundamental Research Funds for the Central Universities 2019NTST17.

\end{document}